\documentclass[11pt,letterpaper]{amsart}

\usepackage{latexsym} 
\usepackage{float}
\usepackage{framed}
\usepackage{textcomp}
\usepackage{amstext}
\usepackage{multicol} 
\usepackage{longtable} 
\usepackage{lscape} 
\usepackage{multirow} 
\usepackage{array} 
\usepackage{amssymb}
\usepackage{amsmath} 
\usepackage{amsthm}
\usepackage[T1]{fontenc} 
\usepackage{graphicx,tikz}
\usepackage{verbatim}
\usepackage[colorlinks=true,linkcolor=blue,citecolor=black,urlcolor=cyan]{hyperref}

\numberwithin{equation}{section}

\newtheorem{theorem}{Theorem}[section]
\newtheorem{proposition}[theorem]{Proposition}
\newtheorem{lemma}[theorem]{Lemma}
\newtheorem{corollary}[theorem]{Corollary}

\theoremstyle{remark}
\newtheorem{remark}[theorem]{\bf Remark}
\theoremstyle{plain}
\newtheorem{prop}[theorem]{Proposition}

\newcommand{\R}{\mathbb R}

\newcommand{\dd}{\,d}
\newcommand{\eps}{\varepsilon}

\newcommand{\ol}{\overline}

\title{ \bf Regularity for convex viscosity solutions of $\sigma_3$ equation}
\thanks{This work was supported by NSFC (No. 12141103 and No. 12071017)}

\author[R. Chen]{Ruosi Chen}
\address{Department of Mathematical Sciences, Tsinghua University, Beijing 100084, China}
\email{crs22@mails.tsinghua.edu.cn}

\author{YanNan Liu}
\address{School of Mathematics and Statistics, Beijing Technology and Business University, Beijing 100048, China}
\email{liuyn@th.btbu.edu.cn}

\author[X. Zhou]{Xingchen Zhou}
\address{School of Mathematics and Statistics, Hainan University, Haikou 570228, China}
\email{zxc3zxc4zxc5@stu.xjtu.edu.cn}

\begin{document}

\begin{abstract}
We prove interior $C^2$ regularity for convex viscosity solutions of the $3$-Hessian equation $\sigma_3(D^2u)=f(x)$ with $f\in C^{0,1}, \inf f>0$, under a strict $3$-convexity condition on $u$.
\end{abstract}

\maketitle

\section{Introduction}

In this note we study interior regularity for convex viscosity solutions of
\begin{equation}\label{eq:main-equation}
        \sigma_k(D^2u)=f(x)>0
\end{equation}
focusing on the case $k=3$. Convexity alone does not rule out Pogorelov-type counterexamples. In fact, in a small neighborhood of the origin, Pogorelov's example
\[
        U(x_1,\ldots,x_k)=(1+x_1^2)(x_2^2+\cdots+x_k^2)^{1-1/k}
\]
gives a convex viscosity solution of a $k$-dimensional Monge--Amp\`ere equation with positive right-hand side; see Pogorelov \cite{Pogorelov}.  Its trivial extension to $\mathbb R^n$ gives a convex viscosity solution of the corresponding $k$-Hessian equation in a neighborhood of the origin in $\mathbb R^n$, for $3\le k\le n$. See also Urbas \cite{Urbas1990}. The contact set of this example contains
\[
    \{x_2=\cdots=x_k=0\}\times \mathbb R^{n-k},
\]
which has dimension $n-k+1$.  
Following Mooney \cite{MR4246798}, we impose the following strict
$3$-convexity condition: for every supporting affine function $L$ of $u$, we require that
\begin{equation}
     \dim \{x:u(x)=L(x)\}\le n-3.
    \tag{C}
    \label{eq:strict-k-convex}
\end{equation}

The main result is the following.

\begin{theorem}
\label{thm:main}
Let $u\in C(B_2)$ be a convex viscosity solution of
\[
        \sigma_3(D^2u)=f(x)\qquad\text{in }B_2,
\]
where $f\in C^{0,1}(B_2)$ and $f \ge f_0>0$.  Assume that $u$ satisfies the strict $3$-convexity condition \eqref{eq:strict-k-convex}. Then $u\in C^2(B_2)$.
\end{theorem}

\begin{remark}
    As the example shows, convexity does not automatically imply strict \(3\)-convexity for solutions of the \(3\)-Hessian equation. This is in contrast with the case \(k=2\), where strict \(2\)-convexity is automatic for convex viscosity solutions with positive right-hand side \cite{MR4246798}. For \(k\ge3\), Chou and Wang \cite{ChouWang2001} showed that strict convexity gives interior \(C^2\) regularity. However, when \(n>3\), strict convexity is stronger than the strict \(3\)-convexity condition.
\end{remark}

The regularity theory for $k$-Hessian equations has been extensively studied.
The classical solvability of Dirichlet problem for admissible solutions of \(k\)-Hessian equations was established by Caffarelli, Nirenberg and Spruck \cite{CaffarelliNirenbergSpruck1985}. 
Urbas \cite{Urbas2001} obtained Hessian bounds in terms of $W^{2,p}$ norm of the solution. 
Chou and Wang \cite{ChouWang2001} developed a variational theory and obtained Pogorelov-type estimates. 
More recently, Zhang \cite{Zhang2025} obtained $C^2$ estimates for semiconvex $k$-admissible solutions under suitable Dirichlet assumptions.

The quadratic Hessian equations $(k=2)$ has received particular attention. Warren and Yuan \cite{WarrenYuan2009} established interior Hessian estimates for the $\sigma_2$ equation in dimension three.
Interior $C^2$ estimates for convex solutions in general dimensions were
obtained by Guan and Qiu \cite{GuanQiu2019}. 
Hessian estimates for convex solutions were also proved by McGonagle, Song and Yuan \cite{McGonagleSongYuan2019} via a compactness argument. 
Shankar and Yuan \cite{ShankarYuan2020} derived interior Hessian estimates for semiconvex solutions, and later established interior regularity for almost convex viscosity solutions \cite{ShankarYuan2021}.
Mooney \cite{MR4246798} proved the strict $2$-convexity of convex viscosity
solutions and, as a consequence, obtained another proof of their interior
regularity. More recently, Qiu \cite{Qiu2024} obtained interior Hessian 
estimates for $\sigma_2$ equations with variable right-hand side in dimension three.
Shankar and Yuan \cite{ShankarYuan2025} established
interior Hessian estimates and regularity in dimension four with constant right hand side. 
Li and Wu \cite{LiWu} extended these results to general dimensions.
Chen, Jian, Tu, and Zhou \cite{ChenJianTuZhou} established interior \(C^2\) regularity for convex viscosity solutions for $\sigma_2$ equation with positive Lipschitz right-hand side.

The proof proceeds by approximation.  We solve Dirichlet problems to obtain smooth admissible solutions. 
A stability estimate gives uniform convergence to the original viscosity solution.
The strict $3$-convexity condition gives a local admissible barrier.  This barrier allows us to apply a boundary Jacobi inequality and the resulting Pogorelov-type estimates to obtain uniform $C^{1,1}$ bounds for the approximating sequence.  The Evans--Krylov--Safonov theory and compactness then yield the desired $C^2$ regularity.

The paper is organized as follows. 
In Section \ref{sec:Jacobi}, we establish the boundary Jacobi inequality.
Section \ref{sec:integral-estimates} gives the Pogorelov-type estimates.
In Section \ref{section:C2-regularity}, we construct approximating solutions, prove uniform convergence, and apply the estimates of Section \ref{sec:integral-estimates} to establish the uniform $C^{1,1}$ bound, thereby proving Theorem \ref{thm:main}.
In Appendix \ref{app:theta}, we give a complete proof of the algebraic lemma used in Lemma \ref{lem:inhomogeneous-theta}.

Throughout this paper, $B_r = B_r(0)$ denotes the open ball of radius $r$ centered at the origin, and $C(\cdots)$ denotes a positive constant depending only on the listed quantities, which may change from line to line.

\section{Boundary Jacobi inequality approach}\label{sec:Jacobi}

Throughout the a priori estimates of the proof, $u$ is a smooth $3$-convex solution of
\begin{equation}\label{eq:smooth-sigma3}
        F(D^2u)=\sigma_3(D^2u)=f(x)>0.
\end{equation}
We write
\[
        \sigma_1:=\sigma_1(D^2u)=\Delta u,
        \qquad \sigma_2:=\sigma_2(D^2u).
\]
Since $\lambda_i\ge 0$, $i=1\cdots n$, we have 
\begin{equation*}
    \sigma_1 \geq c(n) \sigma_3^{1/3}\geq c(n) f_0^{1/3}, \qquad \sigma_2 \geq c(n)\sigma_3^{2/3} \geq c(n) f_0^{2/3}.
\end{equation*}
Denote for $v \in C^2$,
\[
        F_{ij}:=\frac{\partial \sigma_3}{\partial u_{ij}}(D^2u),
        \qquad
        \Delta_F v:=F_{ij}v_{ij},
        \qquad
        |\nabla_F v|^2:=F_{ij}v_i v_j.
\]

At a fixed point, we may assume  $D^2u$ is diagonal and $\lambda_1\ge\lambda_2\ge\cdots\ge\lambda_n$. We write
\[
    d_i:=F_{ii}=\sigma_2(\lambda|i),\qquad d=(d_1,\ldots,d_n)^T,         \qquad e=(1, \ldots,1)^T.
\]
Here $\lambda|i$ means that $\lambda_i$ is omitted.  Since $\lambda\in\Gamma_3$, one has $d_i>0$ for every $i$. Also, 
$$ d_1 \leq \cdots \leq d_n. $$

\begin{lemma}
\label{lem:dynamic-jacobi}
Let $u$ be a smooth $3$-convex solution of \eqref{eq:smooth-sigma3}.  Then there exists $\eps = \eps(n)>0$, such that 
\begin{equation}\label{eq:dynamic-jacobi}
        \Delta_F \Delta u-\frac54\frac{|\nabla_F \Delta u|^2}{\Delta u}
        \ge
        \Delta f-C(n)\frac{|Df|^2}{f},
\end{equation}
whenever
\begin{equation}\label{eq:dynamic-condition}
        d_1\le \eps\,\sigma_2(\lambda).
\end{equation}
\end{lemma}

\begin{proof}
For any given $x_0$, after a rotation, we may assume that $D^2u(x_0)$ is diagonal. We have 
\[
D^2u(x_0)=\mathrm{diag}(\lambda_1,\dots,\lambda_n),\qquad
\lambda_1\ge \cdots \ge \lambda_n,
\qquad
\lambda=(\lambda_1,\dots,\lambda_n)\in \Gamma_{3},
\]
and 
\[
F_{ij}=\frac{\partial\sigma_3}{\partial u_{ij}}=d_i\,\delta_{ij}.
\]
Then 
\begin{align}
\label{eq:DeltaF_D}
\Delta_{F}\Delta u & = \sum_{i,j=1}^n F_{ij} \partial_{ij}\Delta u = \sum_{i=1}^{n} F_{ii} \partial_{ii}\Delta u , \\
\label{eq:nablaF_D}
|\nabla_{F} \Delta u|^{2} & = \sum_{i,j=1}^n F_{ij} (\Delta u_i) (\Delta u_j) =  \sum_{i=1}^{n} F_{ii} (\Delta u_{i})^{2}.
\end{align}
Differentiate equation \eqref{eq:smooth-sigma3} with respect to $x_\ell $, we have 
\begin{equation}
    \label{eq:diff-main-once}
    \sum_{i,j=1}^n \frac{\partial \sigma_3 }{\partial u_{ij}} u_{ij\ell } = \sum^n_{i,j=1}F_{ij}u_{ij\ell } = \sum_{i=1}^n d_i u_{ii\ell } = f_\ell 
\end{equation}
for each $1\le \ell \le n$.
Differentiate equation \eqref{eq:smooth-sigma3} with respect to $x_\ell $ twice, we have
\begin{equation}
    \sum^n_{i,j,r,s} \frac{\partial^2\sigma_3}{\partial u_{ij}\partial u_{rs}} u_{ij\ell } u_{rs\ell }+\sum^n_{i,j=1}F_{ij} u_{ij\ell \ell }=f_{\ell \ell }.
    \label{eq:twice-diff}
\end{equation}
Direct computation gives
$$ 
    \frac{\partial^2\sigma_3 }{\partial u_{ij}\partial u_{rs}} = \begin{cases}
        \sigma_{1}(\lambda|ir), \quad & i=j, r=s, i \neq r; \\
        -\sigma_{1}(\lambda|ij), \quad & i\neq j, r=j, s=i; \\
        0 , \quad &\text{else.}            
    \end{cases}
$$
Summing up the equality \eqref{eq:twice-diff}, we get
\begin{equation*}
    \Delta_F\Delta u-\Delta f
=
-\sum_{\ell=1}^n\sum_{i\neq j}\sigma_{1}(\lambda|ij)\,u_{ii\ell}u_{jj\ell}
+\sum_{\ell=1}^n\sum_{i\neq j}\sigma_{1}(\lambda|ij)\,u_{ij\ell}^2.
\end{equation*}
Notice that 
\begin{eqnarray*}
    \sum_{\ell=1}^n\sum_{i\neq j}\sigma_{1}(\lambda|ij)\,u_{ij\ell}^2
& \ge & \sum_{m=1}^n\sum_{i\neq m}
\Bigl[ \sigma_{1}(\lambda|im)\,u_{imi}^2
+\sigma_{1}(\lambda|mi)\,u_{mii}^2 \Bigr] \\
&= & \sum_{m=1}^n\sum_{i\neq m}2\sigma_{1}(\lambda|im)\,u_{iim}^2.
\end{eqnarray*}
Thus 
\begin{equation}
\label{eq:Delta_F_Delta_u}
\Delta_F\Delta u-\Delta f
\geq \sum_{m=1}^n \Big[ 2\sum_{i:\, i\neq m}\sigma_{1}(\lambda|im)\,u_{iim}^2
- \sum_{i \neq j}\sigma_{1}(\lambda|ij)\,u_{iim}u_{jjm} \Big].
\end{equation}

For each $m\in\{1,\ldots,n\}$, define the symmetric matrix $A^{(m)}(\lambda)$ by
\begin{equation}\label{eq:A-m-def}
        A^{(m)}_{ii}=2\sigma_1(\lambda|im)\quad(i\ne m),\qquad
        A^{(m)}_{mm}=0, \qquad A^{(m)}_{ij}=-\sigma_1(\lambda|ij)\quad(i\ne j).
\end{equation}
Denote
\[
z^{(m)}:=(u_{11m},u_{22m},\dots,u_{nnm})^T.
\]
Then  
\begin{equation}\label{eq:first-derivative}
d^Tz^{(m)} = \sum_{i=1}^n d_iu_{iim}=f_m, \qquad e^Tz^{(m)} = \sum_{i=1}^n u_{iim} = \Delta u_m,
\end{equation}
together with
\begin{equation}\label{eq:gradF}
|\nabla_F\Delta u|^2
=\sum_{m=1}^n d_m\,(\Delta u)_m^2
=\sum_{m=1}^n d_m\,(e^Tz^{(m)})^2.
\end{equation}
Combining \eqref{eq:Delta_F_Delta_u}, we have 
\begin{equation}\label{eq:reduce-to-A}
\Delta_F\Delta u-\delta \frac{|\nabla_F\Delta u|^2}{\Delta u}
\ge \Delta f + \sum_{m=1}^n \Bigg[ (z^{(m)})^T A^{(m)}(\lambda)\,z^{(m)} - \delta \frac{d_m}{\sigma_1} (e^Tz^{(m)})^2\Bigg],
\end{equation}
where $\delta>1$. Applying Lemma \ref{lem:inhomogeneous-theta} with $\delta = 5/4$ and $\bar \delta = 4/3$, there exists $\eps = \eps(n)>0$ such that whenever $d_1 \leq \eps \sigma_2(\lambda)$, we have 
\begin{eqnarray*}
    (z^{(m)})^T A^{(m)}(\lambda)\,z^{(m)} - \frac{5}{4} \frac{d_m}{\sigma_1} (e^Tz^{(m)})^2 \geq -C(n) \frac{(d^Tz^{(m)})^2}{\sigma_3(\lambda)} = -C(n) \frac{f_m^2}{f}.
\end{eqnarray*}
Plugging into \eqref{eq:reduce-to-A}, we get
$$\Delta_F \Delta u-\frac54\frac{|\nabla_F \Delta u|^2}{\Delta u}  \ge  \Delta f-C(n)\frac{|Df|^2}{f},$$
which proves \eqref{eq:dynamic-jacobi}. This completes the proof.
\end{proof}

Now we prove Lemma \ref{lem:inhomogeneous-theta} which was used above. 

\begin{lemma}[Inhomogeneous $\Theta$ inequality]\label{lem:inhomogeneous-theta}
Fix numbers
\[
        1<\delta<\bar\delta<\frac32.
\]
Then there exist $\eps=\eps(n,\bar\delta)>0$ and $C=C(n,\delta,\bar\delta)>0$ such that, whenever $\lambda\in\Gamma_3$ is ordered and satisfies
\[
        d_1\le \eps\,\sigma_2(\lambda),
\]
then for every $m$ and every $z\in\R^n$,
\begin{equation}\label{eq:inhomogeneous-theta}
        z^TA^{(m)}z
        \ge
        \delta\frac{d_m}{\sigma_1}(e^Tz)^2
        -C\frac{(d^Tz)^2}{\sigma_3(\lambda)}.
\end{equation}
\end{lemma}

\begin{proof}
Since
\[
\frac{\partial^2\sigma_3}{\partial \lambda_i\partial \lambda_j}(\lambda)
=
\begin{cases}
\sigma_1(\lambda|ij), & i\ne j,\\
0, & i=j.
\end{cases}
\]
Therefore
\[
D^2\sigma_3(\lambda)[z,z]
=
\sum_{i,j=1}^n
\frac{\partial^2\sigma_3}{\partial \lambda_i\partial \lambda_j}(\lambda)z_iz_j
=
\sum_{i\ne j}\sigma_1(\lambda|ij)z_iz_j.
\]
Then, for $z\in\R^n$,
\begin{equation}\label{eq:A-identity}
        z^T A^{(m)}z
        =-D^2\sigma_3(\lambda)[z,z]
          +2\sum_{i\ne m}\sigma_1(\lambda|im)z_i^2.
\end{equation}
Since \(\tilde{F}:= \sigma_3^{1/3}\) is concave on \(\Gamma_3\), for every direction \(z\in \R^n\) we have
\[
D^2\tilde F(\lambda)[z,z]\le 0.
\]
From the definition of $d$, we have 
\[
D\sigma_3(\lambda)[z]= \sum_{i=1}^n \frac{\partial \sigma_3}{\partial\lambda_i} z_i = d^Tz.
\]
By differentiating \(\tilde F=\sigma_3^{1/3}\), we get
\[
D^2\tilde F(\lambda)[z,z]
=
\frac{1}{3}\sigma_3(\lambda)^{-2/3}
D^2\sigma_3(\lambda)[z,z]
-
\frac{2}{9}\sigma_3(\lambda)^{-5/3}
\bigl(D\sigma_3(\lambda)[z]\bigr)^2.
\]
Using \(D^2\tilde F(\lambda)[z,z]\le 0\), it follows that
\begin{equation}\label{eq:hodge}
        D^2\sigma_3(\lambda)[z,z]
        \le \frac{2}{3}\frac{(d^Tz)^2}{\sigma_3(\lambda)}.
\end{equation}
Consequently
\begin{equation}\label{eq:B-psd}
        B^{(m)}:=A^{(m)}+\frac{2}{3\sigma_3(\lambda)}dd^T
        \quad\text{ is positive semidefinite.}
\end{equation}
Define
\begin{equation}\label{eq:I-Theta-def-main}
        I_m(\lambda):=\inf_{d^Tz=0,\ e^Tz=1} z^TA^{(m)}z,
        \qquad
        \Theta_m(\lambda):=\frac{\sigma_1 }{d_m}I_m(\lambda),
        \qquad
        \Theta(\lambda):=\min_m\Theta_m(\lambda).
\end{equation}
Appendix \ref{app:theta} proves the following algebraic theorem: for every fixed
\[
        1<\bar\delta<\frac32
\]
there exists $\eps=\eps(n,\bar\delta)>0$ such that, for ordered $\lambda\in\Gamma_3$,
\begin{equation}\label{eq:theta-theorem-used}
        d_1\le \eps\,\sigma_2(\lambda)
        \quad\Longrightarrow\quad
        \Theta(\lambda)\ge \bar\delta.
\end{equation}

If $d^T z=0, e^Tz = 0$, then \eqref{eq:inhomogeneous-theta} follows immediately from $z^T A^{(m)} z \geq 0$ by \eqref{eq:A-identity} and \eqref{eq:hodge}.

If $ d^Tz =0, e^Tz \neq 0$, then  \eqref{eq:theta-theorem-used} implies
$$ z^T A^{(m)} z \geq \bar\delta \frac{d_m}{\sigma_1} (e^T z)^2 \geq \delta \frac{d_m}{\sigma_1} (e^T z)^2, $$
thus \eqref{eq:inhomogeneous-theta} follows.

Now we can assume $c:= d^T z \neq 0$.  Set
\[
    w:=\frac{e_m}{d_m},\qquad y:=z-cw,
\]
where $e_m = (0,\cdots,\overset{(m)}1,\cdots,0) \in \R^n$. 
Then $d^Tw=1$ and $d^Ty=0$.  Also $w^TA^{(m)}w=0$ since the $(m,m)$ entry of $A^{(m)}$ is zero. Recall that $B^{(m)}$ is positive semidefinite from \eqref{eq:B-psd}. The Cauchy-Schwarz inequality gives 
\begin{eqnarray}
\label{eq:Cauchy-yBw}
    (y^T B^{(m)}w)^2 \leq (y^T B^{(m)}y) \cdot (w^T B^{(m)}w).
\end{eqnarray}
Since $d^Ty=0$, direct computation gives 
$$ y^T B^{(m)}y= y^T A^{(m)}y, \qquad y^T B^{(m)}w=y^T A^{(m)}w, \qquad w^T B^{(m)}w = \frac{2}{3\sigma_3(\lambda)}. $$
Then \eqref{eq:Cauchy-yBw} gives
\begin{eqnarray*}
    (y^T A^{(m)}w)^2 \leq \frac{2}{3\sigma_3(\lambda)} \, y^T B^{(m)}y.
\end{eqnarray*}
Therefore, for every $\eta\in(0,1)$,
\begin{equation}\label{eq:cross-theta-control}
        2c y^T A^{(m)}w
        \ge -\eta\, y^T A^{(m)}y -C(\eta)\frac{c^2}{\sigma_3(\lambda)}.
\end{equation}
Using \eqref{eq:theta-theorem-used} on $y\in\{d^Tz=0\}$,
\[
        y^T A^{(m)}y
        \ge \bar\delta\frac{d_m}{\sigma_1}(e^Ty)^2.
\]
Hence \eqref{eq:cross-theta-control} implies
\begin{eqnarray}
\label{eq:Az-y-bound}
        z^T A^{(m)}z = y^T A^{(m)} y + 2c\, y^T A^{(m)}w \geq  (1-\eta)\bar\delta\frac{d_m}{\sigma_1}(e^Ty)^2
        -C(\eta)\frac{c^2}{\sigma_3(\lambda)}.
\end{eqnarray}
Choose $\eta$ so small that $(1-\eta)\bar\delta>\delta$.  Since
\[
        e^Ty=e^Tz-\frac{c}{d_m},
\]
we may choose $\theta\in(0,1)$ so small that
\[
    (1-\eta)(1-\theta)\bar\delta>\delta.
\]
The elementary inequality $(a-b)^2\ge(1-\theta)a^2-C(\theta) b^2$ gives
\begin{eqnarray}
\label{eq:eTy}
    \frac{d_m}{\sigma_1}(e^Ty)^2
        \ge (1-\theta)\frac{d_m}{\sigma_1}(e^Tz)^2
        -C(\theta)\frac{c^2}{\sigma_1 d_m}.
\end{eqnarray}
Using Newton-Maclaurin inequality, 
\begin{equation}
    \sigma_1 d_m = \sigma_1(\lambda|m) \sigma_2(\lambda|m) + \lambda_m \sigma_2(\lambda|m) \geq \sigma_3(\lambda|m) + \lambda_m \sigma_2(\lambda|m) = \sigma_3(\lambda).
    \label{eq:sigma_1-dm}
\end{equation}
Combining \eqref{eq:Az-y-bound}-\eqref{eq:sigma_1-dm}, we get \eqref{eq:inhomogeneous-theta}, which completes the proof of Lemma \ref{lem:inhomogeneous-theta}.
\end{proof}

\medskip

Before proving boundary Jacobi inequality, we give a lower bound for $\Delta_F \Delta u$.

\begin{lemma}
\label{lem:subharmonicity}
Let $u$ be a smooth $3$-convex solution of \eqref{eq:smooth-sigma3}.  Then
\begin{equation}\label{eq:subharmonicity}
        \Delta_F \Delta u \ge
        \Delta f-\frac23\frac{|Df|^2}{f}.
\end{equation}
\end{lemma}

\begin{proof}
Recalling \eqref{eq:Delta_F_Delta_u} and the definition of $A^{(m)}$ and $z^{(m)}$, we have 
$$ \Delta_F \Delta u - \Delta f \geq \sum_{m=1}^n \, (z^{(m)})^T A^{(m)} z^{(m)}.  $$
By \eqref{eq:hodge} and \eqref{eq:A-identity}, for every $m$ and every $z\in\R^n$,
\[
        z^TA^{(m)}z
        \ge -\frac23\frac{(d^Tz)^2}{\sigma_3(\lambda)}.
\]
Apply this to $z^{(m)}=(u_{11m},\ldots,u_{nnm})^T$. Since $d^Tz^{(m)}=f_m$, we obtain
\[
        \Delta_F \Delta u
        \ge \Delta f-\frac23\frac{|Df|^2}{f}.
\]
\end{proof}

\subsection{Boundary Jacobi inequality}\label{sec:boundary-jacobi}

Set
\begin{equation}\label{eq:b-def}
        b:=\Delta u+A+e^{|Du|^2},
\end{equation}
where $A=A(n,f_0,\|u\|_{C^{0,1}(B_1)})$ is chosen below.  Since $\Delta u\ge c(n)f_0^{1/3}$, after increasing $A$ if necessary, we have
\begin{equation}\label{eq:b-comparable}
        \Delta u \le b\le C \Delta u 
\end{equation}
with $C=C(n,f_0,\|u\|_{C^{0,1}(B_1)})$.

\begin{proposition}[Boundary Jacobi inequality]\label{prop:boundary-jacobi}
Let $u$ be a smooth $3$-convex solution to equation \eqref{eq:smooth-sigma3} in $B_2$ and let $\Omega\subset B_1$ be a connected open subset. Then for any $\phi\in C^2(\Omega)$ satisfying $\phi>0$ in $\Omega$, we have 
\begin{equation}\label{eq:boundary-jacobi}
        \Delta_F(\phi^{6} b)
        +\frac12\frac{|\nabla_F(\phi^6  b)|^2}{\phi^6 b}
        \ge
        6\phi^{5}b\Delta_F\phi
        +\phi^6\Delta f
        -C \Delta u ,
\end{equation}
where $C=C(n,f_0,\|f\|_{C^{0,1}(B_1)},\|u\|_{C^{0,1}(B_1)},\|
\phi\|_{C^{0,1}(\Omega)})$. 
\end{proposition}

\begin{proof}
    Denote $\varphi := \phi^{6}, J:= \Delta_F (\varphi b) + \frac{1}{2} \frac{|\nabla_F(\varphi b)|^2}{\varphi b}$ and 
    \begin{eqnarray*}
    \Omega_1&:=&\left\{x\in \Omega \mid  d_1 \leq \eps \sigma_2 \right\}, \\
    \Omega_2 &:=& \left\{ x\in \Omega \mid \frac{|\nabla_F \varphi|}{\varphi} \leq \frac{1}{6} \frac{|\nabla_F b|}{b} \right\},
    \end{eqnarray*}
    where $\eps = \eps(n)$ comes from Lemma \ref{lem:dynamic-jacobi}. 
    We will prove the target inequality \eqref{eq:boundary-jacobi} in $\Omega_1$, $\Omega_2$ and $\Omega_1^c\cap \Omega_2^c$ respectively. 

    \textbf{Step 1:} In $\Omega_1$. 
    Direct computation shows that 
    \begin{equation}
    \label{eq:nabla_F-varphi-b}
        \begin{aligned}
            \Delta_F (\varphi b) & = \varphi \Delta_F b + 2 \sum_{i,j=1}^nF_{ij} \varphi_i b_j + b\Delta_F \varphi \\
            & \geq \varphi \Delta_F b - \frac{5}{6} \frac{|\nabla_F\varphi|^2}{\varphi} b - \frac{6}{5} \frac{|\nabla_F b|^2}{b}\varphi + b\Delta_F \varphi \\
            & = b \left( \Delta_F \varphi - \frac{5}{6} \frac{|\nabla_F\varphi|^2}{\varphi} \right) + \varphi \left( \Delta_F b - \frac{6}{5} \frac{|\nabla_F b|^2}{b} \right).
        \end{aligned}
    \end{equation}
    Since $\varphi = \phi^6$, we have 
    \begin{equation}
    \label{eq:nabla_F-varphi}
        \Delta_F \varphi - \frac{5}{6} \frac{|\nabla_F\varphi|^2}{\varphi} = 6 \phi^{5} \Delta_F \phi.
    \end{equation}
    Young's inequality yields 
    \begin{equation*}
        |\nabla_F b|^2 = |\nabla_F(\Delta u + e^{|Du|^2})|^2 \leq \left(1+\frac{1}{24}\right)|\nabla_F \Delta u|^2 +  C(n)  |\nabla_F e^{|Du|^2}|^2,
    \end{equation*}
    and thus
    \begin{equation*}
        \Delta_F b - \frac{6}{5} \frac{|\nabla_F b|^2}{b} \geq \Delta_F \Delta u + \Delta_F e^{|Du|^2} - \frac{5}{4} \frac{|\nabla_F \Delta u|^2}{\Delta u} - C(n) \frac{|\nabla_F e^{|Du|^2}|^2}{A}.
    \end{equation*}
    In $\Omega_1$, we can use Lemma \ref{lem:dynamic-jacobi}. As a result, we have 
    \begin{equation*}
        \Delta_F \Delta u - \frac{5}{4} \frac{|\nabla_{F}\Delta u|^{2}}{\Delta u} \geq \Delta f-C(n,f_0,\|f\|_{C^{0,1}}) \quad \text{ in } \Omega_1.
    \end{equation*}
    So we have 
    $$
    \Delta_F b - \frac{6}{5} \frac{|\nabla_F b|^2}{b} \geq \Delta f-C(n,f_0,\|f\|_{C^{0,1}(B_2)}) + \Delta_F e^{|Du|^2} - C(n) \frac{|\nabla_F e^{|Du|^2}|^2}{A}.
    $$
    Notice that 
    $$ |\nabla_F e^{|Du|^2}|^2 = e^{2|Du|^2} |\nabla_F|Du|^2|^2. $$
    Using \eqref{eq:diff-main-once}, we have
    \begin{eqnarray}
        \Delta_F e^{|Du|^2} &=& 2 e^{|Du|^2} \sum_{i,j,k} F_{ij} u_{ki} u_{kj} + 2 e^{|Du|^2} \sum_{i,j,k} u_k F_{ij} u_{ijk} + e^{|Du|^2} |\nabla_F|Du|^2|^2 \notag \\
        &\geq & 0 - C(\|u\|_{C^{0,1}(B_2)}, \|f\|_{C^{0,1}(B_2)}) + e^{|Du|^2} |\nabla_F|Du|^2|^2 .
        \label{eq:Delta_F-exp}
    \end{eqnarray}
    Therefore, taking $A = A(n, f_0, \|u\|_{C^{0,1}(B_1)})$ large enough, we have
    \begin{eqnarray}\label{eq:nabla_F-b}
        \Delta_F b - \frac{6}{5} \frac{|\nabla_F b|^2}{b} \geq \Delta f - C(n,f_0,\|f\|_{C^{0,1}(B_2)}, \|u\|_{C^{0,1}(B_2)}) 
    \end{eqnarray}
    So, in $\Omega_1$, by \eqref{eq:nabla_F-varphi-b} \eqref{eq:nabla_F-varphi} \eqref{eq:nabla_F-b} we have thus proved 
    \begin{equation*}
        \begin{aligned}
            J &\geq 6 \phi^{5} b \Delta_F \phi + \varphi\big[\Delta f - C(n, f_0, \|f\|_{C^{0,1}(B_2)}, \|u\|_{C^{0,1}(B_2)}) \big] \\
            & \geq 6 \phi^{5} b \Delta_F \phi + \varphi \Delta f - C(n, f_0, \|f\|_{C^{0,1}(B_2)}, \|u\|_{C^{0,1}(B_2)},\|\phi\|_{L^{\infty}(B_2)}).
        \end{aligned}
    \end{equation*}

    \textbf{Step 2:} In $\Omega_2$. We have 
    $$ |\sum_{i,j=1}^nF_{ij}\varphi_ib_j| \leq |\nabla_F \varphi| \cdot |\nabla_F b| \leq \frac{\varphi}{6b}  |\nabla_F b|^2 . $$
    where the last inequality is due to the condition of $\Omega_2$. This implies 
    \begin{eqnarray*}
        J &=& \varphi \Delta_F b + 2 \sum_{i,j=1}^nF_{ij} \varphi_i b_j + b\Delta_F \varphi + \frac{1}{2} \left( \frac{b}{\varphi}|\nabla_F \varphi|^2 + 2\sum_{i,j=1}^nF_{ij} \varphi_i b_j + \frac{\varphi}{b}|\nabla_F b|^2  \right) \\
        &\geq & \varphi \Delta_F b + b\Delta_F \varphi + 3 \sum_{i,j=1}^nF_{ij} \varphi_i b_j + \frac{\varphi}{2b} |\nabla_F b|^2 \\
        &\geq & \varphi \Delta_F b + b\Delta_F \varphi.
    \end{eqnarray*}
    From \eqref{eq:nabla_F-varphi}, we know $\Delta_F \varphi \geq 6\phi^{5} \Delta_F \phi$. 
    Combining Lemma \ref{lem:subharmonicity} and \eqref{eq:Delta_F-exp}, we know 
    $$ \Delta_F b = \Delta_F \Delta u + \Delta_F e^{|Du|^2} \geq \Delta f - C(n, f_0, \|u\|_{C^{0,1}(B_2)}, \|f\|_{C^{0,1}(B_2)}) $$
    So, in $\Omega_2$, we have proved
    $$ J \geq 6 \phi^{5} b \Delta_F \phi + \varphi \Delta f - C(n, f_0, \|f\|_{C^{0,1}(B_2)}, \|u\|_{C^{0,1}(B_2)},\|\phi\|_{L^{\infty}(B_2)}) . $$

    \textbf{Step 3:} In $\Omega_1^c \cap \Omega_2^c $. 
    Using \eqref{eq:nabla_F-varphi-b} and \eqref{eq:nabla_F-varphi}, we have 
    $$ J \geq 6 \phi^{5} b \Delta_F \phi +\varphi \left( \Delta_F \Delta u + \Delta_F e^{|Du|^2} - \frac{6}{5} \frac{|\nabla_F b|^2}{b} \right). $$
    From Lemma \ref{lem:subharmonicity}, we have $\Delta_F\Delta u \geq \Delta f - C(n)|Df|^2/f$. 
    To estimate $\Delta_F e^{|Du|^2}$, we may assume $D^2u(x)$ is diagonal at a given point. Then
    \begin{eqnarray*}
        \Delta_F e^{|Du|^2} &=& 2 e^{|Du|^2} \sum_{i,j,k} F_{ij} u_{ki} u_{kj} + 2 e^{|Du|^2} \sum_{i,j,k} u_k F_{ij} u_{ijk} +  e^{|Du|^2} |\nabla_F|Du|^2|^2.
    \end{eqnarray*}
    By the condition of $\Omega_1^c$, 
    \begin{eqnarray*}
        \sum_{i,j,k} F_{ij} u_{ki} u_{kj} = \sum_{i=1}^n d_i \lambda_i^2 \geq d_1 \lambda_1^2 \geq c_1(n) \, \sigma_2(D^2 u) \cdot (\Delta u)^2.
    \end{eqnarray*}
    Combining \eqref{eq:diff-main-once}, we have 
    \begin{eqnarray*}
        \Delta_F e^{|Du|^2} \geq c_1(n) \, \sigma_2(D^2 u) \cdot (\Delta u)^2 - C(n,\|u\|_{C^{0,1}(B_2)}, \|f\|_{C^{0,1}(B_2)}),
    \end{eqnarray*}
    Meanwhile, from the definition of $\Omega_2^c$, we have 
    $$ \frac{|\nabla_F b|^2}{b} \leq 36\, b\, \frac{|\nabla_F\varphi|^2}{\varphi^2} \leq C_1(n, f_0,\|\phi\|_{C^{0,1}(B_2)}, \|u\|_{C^{0,1}(B_2)})\, \sigma_2(D^2 u) \, \Delta u\, \varphi^{-\frac{1}{3}}, $$
    where we use $b \leq C(n, f_0, \|u\|_{C^{0,1}(B_2)}) \Delta u$ and $F_{ii} = \sigma_2(\lambda|i) \leq C(n)\, \sigma_2(D^2 u)$ in the last inequality. Combining these inequalities, we have 
    \begin{eqnarray*}
        J \geq 6 \phi^{5} b \Delta_F \phi +\varphi \Delta f - C + c_1\, \sigma_2 \cdot (\Delta u)^2 \varphi - C_1 \, \sigma_2  \, \Delta u\, \varphi^{ \frac{2}{3}}.
    \end{eqnarray*}
    Notice that $(\Delta u)^2 \geq 2\sigma_2$, there exists $C_2 = C_2(c_1,C_1)>0$ large enough such that 
    \begin{eqnarray*}
        c_1\, \sigma_2 \cdot (\Delta u)^2 \varphi + C_2 \Delta u &=& \frac{c_1}{2}\, \sigma_2 \cdot (\Delta u)^2 \varphi + \frac{c_1}{2} \, \sigma_2 \cdot (\Delta u)^2 \varphi+ C_2 \Delta u \\
        & \geq & 3\left( \frac{c_1^2}{4}C_2\right)^{\frac13} \sigma_2^{\frac23} (\Delta u)^{\frac53} \varphi^{\frac{2}{3}} \\
        & \geq & C_1 \, \sigma_2  \, \Delta u\, \varphi^{ \frac{2}{3}}.
    \end{eqnarray*}
    Thus 
    \begin{eqnarray*}
        J  \geq 6 \phi^{5} b \Delta_F \phi +\varphi \Delta f - C(n, f_0,\|f\|_{C^{0,1}(B_2)}, \|u\|_{C^{0,1}(B_2)},\|\phi\|_{C^{0,1}(B_2)}) \Delta u.
    \end{eqnarray*}
    This completes the proof of inequality \eqref{eq:boundary-jacobi} in the whole region.
\end{proof}

\section{Pogorelov-type estimates for smooth solutions}\label{sec:integral-estimates}

\begin{proposition}[Pogorelov-type $W^{2,p}$ estimate]
\label{prop:integral}
Let $u$ be a smooth $3$-convex solution of \eqref{eq:smooth-sigma3} in $B_2$ with $f\in C^\infty(B_2)$ and $f\geq f_0>0$. Suppose $\Omega\subset B_1$ is a connected open subset in $B_1$, and  $w\in C^2(\R^n)$ is $3$-convex satisfying
\[
        w>u\quad\text{in }\Omega,
        \qquad
        w=u\quad\text{on }\partial\Omega.
\]
Then for $\varphi = (w-u)^6$ and every integer $p\geq 1$, we have 
\begin{equation}\label{eq:integral-estimate}
    \int_\Omega (\Delta u)^p\varphi^{p-1}\dd x \le p! \, C^p,
\end{equation}
where $ C= C(n, f_0, \|f\|_{C^{0,1}(B_2)}, \|u\|_{L^\infty(B_2)}, \|w\|_{C^{0,1}(B_2)})$. 
\end{proposition}

\begin{proof}
    By the gradient estimate by Chou-Wang \cite[Theorem 3.2]{ChouWang2001}, we know that $\|u\|_{C^{0,1}(B_1)}$ can be bounded by a constant depending only on $n, f_0, \|f\|_{C^{0,1}(B_2)}$, and $\|u\|_{L^\infty(B_2)}$.
    Hence, in the proof below, we may use $\|u\|_{C^{0,1}(B_1)}$, and at the end replace this dependence by the above quantities. We denote constant $C = C(n, f_0, \|f\|_{C^{0,1}(B_2)}, \|u\|_{C^{0,1}(B_1)}, \|w\|_{C^{0,1}(B_2)})$ which may change from line to line. 
    Write $\varphi = \phi^6$, where $\phi = w-u \in C^2(\Omega)$. Then 
    \begin{equation}\label{eq:DeltaF-psi-lower}
            \Delta_F\phi=\Delta_Fw-\Delta_Fu\ge -3f,
    \end{equation}
    where we use $\Delta_F w >0$ since $w$ is $3$-convex, and $\Delta_F u = 3f$. 
    Applying Proposition \ref{prop:boundary-jacobi}  gives
    \begin{equation}\label{eq:J-integral-start}
        \Delta_F(\varphi b)+\frac12\frac{|\nabla_F(\varphi b)|^2}{\varphi b}
        \ge \varphi\Delta f-C \Delta u, 
    \end{equation}
    since the term $6\phi^{5}b\Delta_F\phi$ in \eqref{eq:boundary-jacobi} can be  absorbed into $-C \Delta u$ using \eqref{eq:DeltaF-psi-lower} and $b\le C\Delta u$.
    
    For any integer $p\geq 1$, multiply \eqref{eq:J-integral-start} by $(\varphi b)^p$ and integrate over $\Omega$.  Since $\varphi=0$ on $\partial\Omega$, integration by parts gives
    \begin{equation}
    \label{eq:gradient-main-integral}
        \left(p-\frac12\right)
        \int_\Omega |\nabla_F(\varphi b)|^2(\varphi b)^{p-1}\dd x
        \le
        C\int_\Omega b^p\varphi^p \Delta u\dd x
        -\int_\Omega b^p\varphi^{p+1}\Delta f\dd x.
    \end{equation}
    We choose a small constant $0<\delta<1$ and integrate by parts to estimate that
    \begin{eqnarray}
         \int_{\Omega} b^p \varphi^p \Delta u \,dx &= &-\int_{\Omega} D[(\varphi b)^p] Du \,dx\notag  \\
         &\leq & \delta p\int_{\Omega} |D(\varphi b)|^2 b^{p-2}\varphi^{p-1} \,dx +\frac p{\delta} \int_{\Omega}|Du|^2 b^{p} \varphi^{p-1} \,dx, \label{eq:recursion_ahead}
    \end{eqnarray}
    and
    \begin{eqnarray}
        -\int_{\Omega}\Delta f \varphi^{p+1} b^{p} dx &=& \int_{\Omega}\varphi Df D[(\varphi b)^{p}]\, dx + \int_{\Omega} (\varphi b)^{p}Df D\varphi \,dx \notag \\ 
        & \leq & \delta p\int_{\Omega} |D(\varphi b)|^2 b^{p-2} \varphi^{p-1} dx
        +\frac{pC}{\delta} \int_{\Omega} b^{p}\varphi ^{p-1}dx .\label{eq:temp_Delta_f}
    \end{eqnarray}
    Notice that $F_{ii}\,\Delta u = \sigma_2(\lambda|i) \sigma_1 \geq \sigma_3 \geq f_0 $ for each $1\le i\le n$, thus
    \begin{equation}
        \label{eq:nabla_F_to_D}
        |\nabla_F(\varphi b)|^2(\varphi b)^{p-1}\ge |\nabla_F(\varphi b)|^2\Delta u \varphi^{p-1} b^{p-2} \geq f_0 |D(\varphi b)|^2 b^{p-2}\varphi^{p-1} .
    \end{equation}
    Substituting inequality \eqref{eq:nabla_F_to_D} into \eqref{eq:recursion_ahead} and \eqref{eq:temp_Delta_f}, and then inserting the resulting inequality into \eqref{eq:gradient-main-integral}, we obtain
    \begin{equation*}
        \left(p- \frac{1}{2}\right) \int_{\Omega}|\nabla_F(\varphi b)|^2[\varphi b]^{p-1} dx \leq C\delta p \int_{\Omega}|\nabla_F(\varphi b)|^2[\varphi b]^{p-1}\,dx + \frac{pC}{\delta} \int_{\Omega} b^p \varphi^{p-1}\,dx.
    \end{equation*}
    Let $\delta$ sufficiently small, recall $p\ge 1$, we get 
    \begin{equation*}
        \int_{\Omega}|\nabla_F(\varphi b)|^2[\varphi b]^{p-1} dx \leq C \int_{\Omega} b^p \varphi^{p-1} \,dx.
    \end{equation*}
    Combining \eqref{eq:recursion_ahead}, we can get a recursion formula
    \begin{eqnarray*}
        \int_{\Omega} b^{p+1} \varphi^p \,dx \leq C \int_{\Omega} b^p \varphi^p \Delta u \,dx & \leq& \delta p\int_{\Omega} |D[\varphi b]|^2 b^{p-2}\varphi^{p-1} dx +\frac p{\delta} \int_{\Omega}|Du|^2 b^{p} \varphi^{p-1} dx \\
        & \leq & C\delta p \int _{\Omega}|\nabla_F(\varphi b)|^2[\varphi b]^{p-1} \,dx + pC \int_{\Omega} b^p \varphi^{p-1} \,dx \\
        & \leq & p C \int_{\Omega} b^p \varphi^{p-1} \,dx.
    \end{eqnarray*}
    Therefore,
    $$
    \int_{\Omega} b^{p+1}\varphi ^pdx \le pC \int_{\Omega} b^{p}\varphi ^{p-1}dx \leq \cdots 
    \leq p! C^p \int_{\Omega} b^2 \varphi \,dx 
    \leq p! C^p \int_{\Omega} b \,dx \leq p! C^p \int_{\Omega} \Delta u \,dx.
    $$
   Now we choose a cutoff function $\xi \in C_0^{\infty}(B_2),\xi =1 $ in $B_1$, $|D\xi|<C(n)$, then 
    \begin{eqnarray*}
        \int_\Omega \Delta u \,dx \leq \int_{B_2} \xi \operatorname{div}(Du) \, dx \leq |\int_{B_2} {D\xi \cdot Du} \,dx| \leq C.
    \end{eqnarray*}
\end{proof}

\begin{proposition}[Pogorelov-type $C^{1,1}$ estimate]
\label{prop:pointwise-pogorelov-three}
Let \(u\) be a smooth \(3\)-convex solution of equation \eqref{eq:smooth-sigma3} in $B_2$ with $f\in C^\infty(B_2)$ and $f \geq f_0 >0$.  Let $\Omega\subset B_1$ be a connected open set, and let \(\varphi\in C^2(B_2)\) satisfy \(\varphi\ge 0\) in \(B_2\) and \(\varphi=0\) on \(B_2 \setminus\Omega\).
Then
\[
        \bigl\|\varphi^{3n}\Delta u\bigr\|_{L^\infty(\Omega)}
        \le
        C\int_\Omega \varphi^{2n}(\Delta u)^{2n+1}\,dx,
\]
where $C=C (n,f_0,\|f\|_{C^{0,1}(B_2)}, \|\varphi\|_{C^{0,1}(\Omega)} )$. 
\end{proposition}

\begin{proof}
In this proof, we denote $C=C (n,f_0,\|f\|_{C^{0,1}(B_2)}, \|\varphi\|_{C^{0,1}(\Omega)} )$ which may change from line to line. 
The elementary cone estimates for \(\Gamma_3\), together with
\(\sigma_3(D^2u)=f\ge f_0\), imply
\begin{equation}\label{eq:prop42-ellipticity-deltau}
        F^{ij}\zeta_i\zeta_j
        \ge \frac{c}{\Delta u}|\zeta|^2,
        \qquad
        F^{ij}\zeta_i\zeta_j
        \le C(\Delta u)^2|\zeta|^2,
        \qquad
        \Delta u\ge c
\end{equation}
for every \(\zeta\in\mathbb R^n\). 

Let \(1\le r<R\le2\).  Choose
\(\xi\in C_0^\infty(B_R)\) such that
\[
    0\le\xi\le1, \qquad \xi=1\quad \text{in }B_r,  \qquad |D\xi|\le \frac{2}{R-r} .
\] 
Fix $p\ge 1,q\ge 2$. According to Lemma \ref{lem:subharmonicity}, we have 
$$  \int_{B_R}(\Delta u)^p\varphi^{q}\xi^2\Delta_F\Delta u \, dx \ge\int_{B_R}(\Delta u)^p \varphi^{q}\xi^2 [{\Delta f}-C |Df|^2] dx.  $$
Integrating by parts, using \eqref{eq:prop42-ellipticity-deltau} and Young's inequality repeatedly, we get
\begin{equation*}
    \int_{B_R} |D\Delta u|^2(\Delta u)^{p-2}\varphi^q\xi^2\,dx \le  \frac{C(p+q)^2}{(R-r)^2} \int_{B_R}(\Delta u)^{p+3}\varphi^{q-2}\,dx .
\end{equation*}
Then we have
\begin{align*}
    \int_{B_r}|D[(\Delta u)^{\frac p2}\varphi^{\frac q2}]|^2dx
    \le \frac{C (p+q)^4}{(R-r)^2}\int_{B_{R}}(\Delta u)^{p+3}\varphi^{q-2}\,dx,
\end{align*}
and
\begin{align*}
    \int_{B_{r}}|(\Delta u)^{\frac p2}\varphi^{\frac q2}|^2 dx
    \le \frac{C (p+q)^4}{(R-r)^2}\int_{B_{R}}(\Delta u)^{p+3}\varphi^{q-2}\,dx .
\end{align*}
These two estimates give $(\Delta u)^{\frac p2}\varphi^{\frac q2} \in W^{1,2}(B_{r})$. 
Since $n\geq 3$, we have $ W^{1,2}(B_r) \hookrightarrow L^{\frac{2n}{n-2}}(B_r) $ by Sobolev embedding theorem, and therefore
\begin{equation}\label{eq:prop42-sobolev-step-deltau}
    \left(
    \int_{B_r}
    (\Delta u)^{\gamma p}\varphi^{\gamma q}\,dx
    \right)^{1/\gamma}
    \le
    \frac{C(p+q)^4}{(R-r)^2}
    \int_{B_R}(\Delta u)^{p+3}\varphi^{q-2}\,dx ,
\end{equation}
where $\gamma = \frac{n}{n-2}$.
Fix an integer $k_0\ge\frac{\ln 2n}{\ln\gamma}$ and take $p_0=\gamma^{k_0}, q_0 = 3np_0$ to initiate the iteration. For $k = 1,\cdots,k_0$, let
\begin{eqnarray*}
    p_{k} &=& \gamma^{-1} p_{k-1}+3=\gamma^{-k}p_0+3\sum^{k-1}_{i=0}\gamma^{-i}, \\
    q_{k} &=& \gamma^{-1} q_{k-1}-2=\gamma^{-k}q_0-2\sum^{k-1}_{i=0}\gamma^{-i}, \\
    r_k &=& 1 + \sum^{k}_{i=1}2^{-(k_0-i+2)}.
\end{eqnarray*}
Since $\sum_{\ell=0}^{\infty}\gamma^{-\ell}=\frac{n}{2}$, we have
\begin{equation}\label{eq:prop42-terminal-pq-deltau}
        p_{k_0}\le 1+\frac{3n}{2}\leq 2n+1,
        \qquad
        q_{k_0}\ge 3n-n=2n .
\end{equation}
Furthermore,
\begin{equation}\label{eq:prop42-pq-growth-deltau}
        p_k+q_k\le C(n)\gamma^{k_0-k}
        \qquad (1\le k\le k_0).
\end{equation}
Apply \eqref{eq:prop42-sobolev-step-deltau} with
\[ r=r_{k-1}, \qquad R=r_k, \qquad p=p_k-3, \qquad q=q_k+2 . \]
Define
\[
    I_k:=\int_{B_{r_k}} (\Delta u)^{p_k}\varphi^{q_k}\,dx .
\]
Then \eqref{eq:prop42-sobolev-step-deltau} becomes
\begin{equation*}
        I_{k-1} \le \left[ C2^{2(k_0-k+2)}(p_k+q_k)^4 I_k \right]^\gamma .
\end{equation*}
Therefore, an iteration process implies
\begin{eqnarray}
    I_0 \le
        C^{\sum_{k=1}^{k_0}\gamma^k}
        2^{2\sum_{k=1}^{k_0}(k_0-k+2)\gamma^k}
        \prod_{k=1}^{k_0}(p_k+q_k)^{4\gamma^k}
        I_{k_0}^{\gamma^{k_0}} .
        \label{eq:prop42-iteration-product-deltau}
\end{eqnarray}
Using \eqref{eq:prop42-pq-growth-deltau}, 
\[
        \prod_{k=1}^{k_0}(p_k+q_k)^{4\gamma^k}
        \le
        C^{\sum_{k=1}^{k_0}\gamma^k}
        \gamma^{4\sum_{k=1}^{k_0}(k_0-k)\gamma^k} .
\]
Since \(p_0=\gamma^{k_0}\),
\[
        \frac1{p_0}\sum_{k=1}^{k_0}\gamma^k\le C(n),
        \qquad
        \frac1{p_0}\sum_{k=1}^{k_0}(k_0-k+2)\gamma^k\le C(n),
\]
and hence taking the power \(1/p_0\) in \eqref{eq:prop42-iteration-product-deltau} yields 
\begin{align*}
    \|\varphi^{3n}\Delta u\|_{L^{p_0} (B_{1})} = \left[\int_{B_{1}}(\Delta u)^{p_0}\varphi^{q_{0}}dx\right]^{1/p_0} = I_0^{1/p_0} \leq C  \int_{B_{2}}(\Delta u)^{p_{k_0}}\varphi^{q_{k_0}}dx.
\end{align*}
Recalling \eqref{eq:prop42-terminal-pq-deltau}, the above estimate implies 
$$ \|\varphi^{3n}\Delta u\|_{L^{p_0} (B_{1})} \leq C \int_{B_{2}}(\Delta u)^{2n+1}\varphi^{2n}dx. $$
Let $k_0\rightarrow +\infty$, that is $p_0 \to +\infty$, we complete the proof of the proposition.
\end{proof}

Proposition \ref{prop:smooth-pogorelov} follows immediately from Proposition \ref{prop:pointwise-pogorelov-three} and Proposition \ref{prop:integral} with $p=2n+1$.

\begin{proposition}
\label{prop:smooth-pogorelov}
Let $u$ be a smooth $3$-convex solution of \eqref{eq:smooth-sigma3} in $B_2$ with $f\in C^\infty(B_2)$ and $f \geq f_0 >0$. Let  $w \in C^2(\mathbb{R}^n)$ be a $3$-convex function, and let $\Omega$ be a connected component of $\{ w >u \}$ such that $\overline\Omega \subset B_1$.  Define $$ \varphi(x):=\bigl((w(x)-u(x))^+\bigr)^6,\qquad x\in B_2. $$
Then 
$$ \| \varphi^{3n} \Delta u\|_{L^\infty(\Omega)} \leq C(n,f_0,\|w\|_{C^{0,1}(B_2)},\|u\|_{C^{0,1}(B_2)},\|f\|_{C^{0,1}(B_2)}). $$
\end{proposition}

\section{Interior \texorpdfstring{$C^{2}$}{} regularity}
\label{section:C2-regularity}
In this section we prove Theorem \ref{thm:main}. Assume that $u$ is a convex viscosity solution to 
\begin{equation*}
    \sigma_{3}(D^2 u) = f(x), \quad x\in B_2,
\end{equation*}
with $f\in C^{0,1}(B_2)$, and $\inf_{B_2} f \geq f_0>0$. 
The proof consists of two main parts:
\begin{enumerate}
    \item Construct a sequence of smooth approximating solutions $\{v_\ell\}$ such that $v_\ell \to u$ uniformly;
    \item Prove the uniform $C^{1,1}$ estimates for $\{v_\ell\}$.
\end{enumerate}

\subsection{Approximating solutions and convergence}

Let $u$ be a convex viscosity solution of \eqref{eq:main-equation}.  Choose smooth approximations $f_\ell\in C^\infty(B_1)$ and $u_\ell\in C^\infty(\overline{B_1})$ such that
\begin{equation}\label{eq:approx-data}
        \|f_\ell - f\|_{L^\infty(B_1)} \leq \frac{f_0}{2\ell^2}, \qquad \|f_\ell\|_{C^{0,1}(B_1)} \leq C\|f\|_{C^{0,1}(B_2)}, \qquad
        \|u_\ell - u\|_{C^0(\overline{B_1})} \leq \frac{1}{\ell},
\end{equation}
for each $\ell\geq 1$. 
For every $\ell$, consider the Dirichlet problem 
\begin{equation}
\label{eq:Dirichlet-approx}
    \begin{cases}
    \sigma_3(D^2 v_\ell) = f_\ell & \text{in}\ B_1, \\
    v_\ell = u_\ell & \text{on}\ \partial B_1.
    \end{cases}
\end{equation}
By the classical result of Caffarelli--Nirenberg--Spruck \cite{CNS}, problem \eqref{eq:Dirichlet-approx} admits a unique smooth $3$-convex solution $v_\ell$. 

To obtain $C^0$ convergence of $v_\ell$ to $u$, we establish the following stability estimate. The argument is standard, see for instance Ishii--Lions \cite{Ishii_LionsViscositysolu}, Trudinger \cite{trudinger1990dirichlet} and Lu--Tsai \cite{LusiyuanPogorelovQuotient}.

\begin{prop}
\label{prop:stability}
    Let $u$ be a convex viscosity solution of $\sigma_3(D^2u) = f$ in $B_1$ and let $v$ be a smooth $3$-convex solution of $\sigma_3(D^2 v) = g$ in $B_1$. Assume $f\in C^0(B_1)$, $f\geq f_0>0$ and $ \|f-g\|_{L^\infty(B_1)} \leq \frac{1}{2}f_0$. Then
    \begin{equation}
    \label{eq:stability}
        \|u-v\|_{C^0(\overline{B_1})} \leq \sup_{\partial B_1}|u-v|+C(n) f_0^{-1/6} \|f-g\|_{L^\infty(B_1)}^{1/2}.
    \end{equation}
\end{prop}

\begin{proof}
    Put $a:=\|f-g\|_{L^\infty(B_1)}$ and choose
    \[
        \eps:=A(n)f_0^{-1/6}a^{1/2},
    \]
    where $A(n)>1$ will be fixed below. If $a=0$, we take an arbitrary $\eps>0$ and let $\eps\to0$ at the end.

    Fix $x_0\in B_1$ and suppose $\varphi\in C^2(B_1)$ touches $u+\eps|x|^2$ from above at $x_0$, i.e.
    \[
        \varphi(x_0)=u(x_0)+\eps|x_0|^2,
        \qquad
        \varphi\ge u+\eps|x|^2 \quad \text{in }B_1.
    \]
    Then $\psi:=\varphi-\eps|x|^2$ touches $u$ from above at $x_0$. Since $u$ is convex, $D^2\psi(x_0)\ge0$. The subsolution condition gives
    \[
        \sigma_3(D^2\psi(x_0))\ge f(x_0).
    \]
    Using
    \[
        \sigma_3(M+2\eps I)=\sigma_3(M)+2(n-2)\eps\sigma_2(M)
        +2(n-1)(n-2)\eps^2\sigma_1(M)+8\binom n3\eps^3,
    \]
    and the Newton--Maclaurin inequality $\sigma_1(D^2\psi(x_0))\ge c(n)f_0^{1/3}$, we obtain
    \[
        \sigma_3(D^2\varphi(x_0)) = \sigma_3(D^2\psi(x_0) + 2\eps I)
        \ge f(x_0)+c(n)\eps^2 f_0^{1/3}.
    \]
    Choosing $A(n)$ sufficiently large, $c(n)\eps^2f_0^{1/3}\ge 2a$. Hence $u+\eps|x|^2$ is a viscosity subsolution of
    \[
        \sigma_3(D^2 w)=f+2a.
    \]

    Define
    \[
        w(x):=u(x)+\eps(|x|^2-1)-\sup_{\partial B_1}|u-v|.
    \]
    Then $w\le v$ on $\partial B_1$. We claim that $w\le v$ in $B_1$. If not, $\max_{B_1}(w-v)=c_0>0$ is attained at an interior point $y_0$, and
    \[
        (w-c_0)(y_0)=v(y_0),\qquad w-c_0\le v\quad\text{in }B_1.
    \]
    Thus $v$ touches $w-c_0$ from above at $y_0$. Since $w$ is a subsolution of $\sigma_3=f+2a$, we get
    \[
        g(y_0)=\sigma_3(D^2v(y_0))\ge f(y_0)+2a,
    \]
    contradicting $g\le f+a$. Therefore
    \begin{equation}\label{eq:upper_u-v}
        u-v\le \eps+\sup_{\partial B_1}|u-v|.
    \end{equation}

    Similarly, define
    \[
        \widetilde w(x):=v(x)+\eps(|x|^2-1)-\sup_{\partial B_1}|u-v|.
    \]
    We claim that $\widetilde w\le u$ in $B_1$. If not, $\widetilde w-u$ attains a positive maximum at an interior point $\widetilde y_0$, and for some $\widetilde c_0>0$,
    \[
        (\widetilde w-\widetilde c_0)(\widetilde y_0)=u(\widetilde y_0),\qquad
        \widetilde w-\widetilde c_0\le u\quad\text{in }B_1.
    \]
    Then $\widetilde w-\widetilde c_0$ touches $u$ from below at $\widetilde y_0$. Since $u$ is a viscosity supersolution of $\sigma_3=f$,
    \begin{equation}\label{eq:tilde-w}
        \sigma_3(D^2\widetilde w(\widetilde y_0))\le f(\widetilde y_0)
        \le g(\widetilde y_0)+a.
    \end{equation}
    On the other hand, $g\ge f-\frac12f_0\ge\frac12f_0$ and $D^2v\in\Gamma_3$, so $\sigma_1(D^2v)\ge c(n)f_0^{1/3}$. Therefore
    \[
        \sigma_3(D^2\widetilde w)=\sigma_3(D^2v+2\eps I)
        \ge g+c(n)\eps^2f_0^{1/3}
        \ge g+2a,
    \]
    which contradicts \eqref{eq:tilde-w}. Hence
    \begin{equation}\label{eq:lower_u-v}
        u-v\ge -\eps-\sup_{\partial B_1}|u-v|.
    \end{equation}
    Combining \eqref{eq:upper_u-v}, \eqref{eq:lower_u-v}, and the definition of $\eps$ gives \eqref{eq:stability}.
\end{proof}

Applying Proposition \ref{prop:stability} to $u$ and $v_\ell$ and using \eqref{eq:approx-data}, we obtain the following uniform convergence.  

\begin{corollary}
\label{coro:uniform_v_k}
    For all $\ell\geq 1$, 
    \begin{eqnarray*}
    \|v_\ell - u\|_{C^0(\overline{B_1})} \leq \|u_\ell - u\|_{C^0(\overline{B_1})} +  f_0^{-1/6} \|f_\ell - f\|_{L^\infty(B_1)}^{1/2} \leq \frac{C}{\ell},
\end{eqnarray*}
where $C$ depends only on $n$ and $f_0$. Consequently, $v_\ell \to u$ uniformly in $B_1$.
\end{corollary}

\subsection{Construction of 3-convex cutoff}

To apply Proposition \ref{prop:smooth-pogorelov} and obtain a uniform $C^{1,1}$ estimate for $v_\ell$, we need a $3$-convex barrier function $w$. 
However, the existence for such a barrier is not a common property for general convex viscosity solution to equation \eqref{eq:main-equation}.
So we impose an additional strict $3$-convex condition \eqref{eq:strict-k-convex}, and we will show in the next two lemmas that condition \eqref{eq:strict-k-convex} ensures the existence of a $3$-convex barrier function.

\begin{lemma}
    \label{lem:cutoff-pre}
    Let $u\in C(\overline{B_1})$ be a convex function satisfying $u(0)=0$ and $u\ge0$. Denote
    \[
        E:=\{x\in \overline{B_1}:u(x)=0\}.
    \]
    If $\dim E\le n-3$, then for every $r>0$ there exists $\eta=\eta(u,r)>0$ such that
    \[
        \{x\in \overline{B_1}:u(x)<\eta\}\subset E_r,
    \]
    where $E_r$ denotes the $r$-neighborhood of $E$.
\end{lemma}

\begin{proof}
    Suppose to the contrary. Then there exist $r_0>0$ and points $x_j\in \overline{B_1}$ such that $u(x_j)<1/j$ but $\operatorname{dist}(x_j,E)\ge r_0$. After passing to a subsequence, $x_j\to x_*\in \overline{B_1}$. Then $u(x_*)=0$, so $x_*\in E$. This contradicts
    \[
        \operatorname{dist}(x_*,E)=\lim_{j\to\infty}\operatorname{dist}(x_j,E)\ge r_0.
    \]
\end{proof}

\begin{lemma}
    \label{lem:barrier-construct}
    Let $u\in C(\overline{B_1})$ be a convex function satisfying the strict $3$-convexity condition \eqref{eq:strict-k-convex}. Then there exist a $3$-convex function $w\in C^\infty(\mathbb R^n)$, constants $\kappa,r>0$, and a connected open set $\Omega\subset B_{3/4}$ containing the origin such that
    \begin{itemize}
        \item[1.] $w>u$ in $\Omega$ and $w=u$ on $\partial\Omega$;
        \item[2.] $B_r\subset\Omega$ and $w-u\ge \kappa$ in $B_r$.
    \end{itemize}
\end{lemma}

\begin{proof}
    After subtracting a supporting affine function at the origin, we may assume
    \[
        u(0)=0,\qquad u\ge0\quad\text{in }B_1.
    \]
    By \eqref{eq:strict-k-convex}, $E:=\{u=0\}$ has dimension at most $n-3$. After a rotation, write $x=(y,z)$ with $y\in\mathbb R^3$ and $z\in\mathbb R^{n-3}$ so that
    \[
        E\subset\{y=0\}=\{0\}\times\mathbb R^{n-3}.
    \]
    Consider
    \[
        W(y,z):=M|y|^2-|z|^2,
    \]
    where $M=M(n)$ is chosen large enough that
    \[
        D^2W=2\operatorname{diag}(M,M,M,-1,\ldots,-1)\in\Gamma_3.
    \]
    Choose $r_0>0$ such that
    \[
        r_0^2\le \frac{1}{4(M+1)}.
    \]
    By Lemma \ref{lem:cutoff-pre}, there exists $\eta_0=\eta_0(u,r_0)>0$ such that
    \begin{equation}\label{eq:u<delta-tube-r0}
        \{x\in\overline{B_1}:u(x)<\eta_0\}\subset\{|y|<r_0\}.
    \end{equation}
    Define
    \[
        w(x):=\eta_0\left(W(x)+\frac14\right)
        =\eta_0\left(M|y|^2-|z|^2+\frac14\right).
    \]
    Then $w$ is $3$-convex. Let
    \[
        D:=B_{3/4}\cap\{|y|<r_0\}.
    \]
    Write $\partial D=S_1\cup S_2$, where
    \[
        S_1:=\partial B_{3/4}\cap\{|y|<r_0\},\qquad
        S_2:=\overline{B_{3/4}}\cap\{|y|=r_0\}.
    \]
    On $S_1$, $|z|^2=9/16-|y|^2$, and hence
    \[
        w=\eta_0\left((M+1)|y|^2-\frac{5}{16}\right)
        \le -\frac{1}{16}\eta_0<0\le u.
    \]
    On $S_2$, \eqref{eq:u<delta-tube-r0} gives $u\ge\eta_0$, while
    \[
        w\le \eta_0\left(Mr_0^2+\frac14\right)<\eta_0\le u.
    \]
    Thus $w<u$ on $\partial D$. Since $w(0)=\eta_0/4>u(0)=0$, the connected component $\Omega$ of $\{w>u\}\cap D$ containing the origin satisfies $\Omega\subset B_{3/4}$, $w>u$ in $\Omega$, and $w=u$ on $\partial\Omega$. By continuity, after decreasing $r>0$ if necessary, there is $\kappa>0$ such that $B_r\subset\Omega$ and $w-u\ge\kappa$ in $B_r$.
\end{proof}

\begin{remark}
\label{remark4.5}
    The constants $r$ and $\kappa$ in Lemma \ref{lem:barrier-construct} are constants attached to the particular solution $u$. They are not controlled by $n$, $f_0$, $\|u\|_{L^\infty}$, and $\|f\|_{C^{0,1}}$.  More concretely, the construction uses a positive lower bound for $u$ away from the zero-contact set; for example, $\kappa$ can be taken explicitly as 
    $$ \kappa = \frac{1}8{} \inf_{\overline B_{3/4}\cap\{|y|=r_0\}}u>0 $$
    when \(r_0\) is fixed.
    Pogorelov's counter-examples show that this positive gap may shrink to zero along a family of convex solutions. This is why the argument below yields regularity for each fixed solution $u$ satisfying the strict $3$-convexity condition, rather than a universal estimate with constants independent of this gap.
\end{remark}

\medskip

\subsection{Proof of Theorem \ref{thm:main}}

From Corollary \ref{coro:uniform_v_k}, $v_\ell\to u$ uniformly in $B_1$.
Let $w$ be the $3$-convex barrier for $u$ given by Lemma \ref{lem:barrier-construct}.  Since $w<u$ on $\partial D$ in the proof of Lemma \ref{lem:barrier-construct}, uniform convergence implies that, for all sufficiently large $\ell$, the connected component $\Omega_\ell$ of $\{w>v_\ell\}$ containing the origin satisfies $\overline{\Omega_\ell}\subset B_{3/4}$. Moreover,
\[
        w-v_\ell\ge \frac\kappa2\quad\text{in }B_r.
\]
The gradient estimate for $k$-Hessian equations with Lipschitz right-hand side \cite[Theorem 3.2]{ChouWang2001} gives
\[
        \|v_\ell\|_{C^{0,1}(B_{3/4})}
        \le C(n,f_0,\|u\|_{L^\infty(B_1)},\|f\|_{C^{0,1}(B_1)}).
\]
Applying Proposition \ref{prop:smooth-pogorelov} after rescaling from $B_1$ to $B_{3/4}$, and using $w-v_\ell\ge\kappa/2$ on $B_r$, we obtain
\begin{equation}\label{eq:v_k_C11_bound}
    \|D^2v_\ell\|_{L^\infty(B_r)}
    \le C(n,\kappa,f_0,\|w\|_{C^{0,1}(B_1)},\|u\|_{L^\infty(B_1)},\|f\|_{C^{0,1}(B_1)}).
\end{equation}
With this $C^{1,1}$ estimate, the equation is uniformly elliptic on the branch $D^2v_\ell\in\Gamma_3$. The Evans--Krylov--Safonov theory gives, for any $\alpha\in(0,1)$,
\[
        \|v_\ell\|_{C^{2,\alpha}(B_{r/2})}
        \le C(n,\alpha,\kappa,f_0,\|w\|_{C^{0,1}(B_1)},\|u\|_{L^\infty(B_1)},\|f\|_{C^{0,1}(B_1)}).
\]
After passing to a subsequence, $v_\ell\to u$ in $C^2(B_{r/2})$. Covering and scaling give $u\in C^2(B_2)$, which completes the proof of Theorem \ref{thm:main}.

\appendix

\section{Appendix}\label{app:theta}

\begin{theorem}
\label{thm:theta-k3}
There exist constants $\eps>0$ and $\delta>1$, depending only on $n$, such that for every ordered $\lambda\in\Gamma_3$,
\[
    d_1\le \eps\sigma_2(\lambda) \quad\Longrightarrow\quad \Theta(\lambda)\ge \delta.
\]
More precisely, for every fixed $\delta$ with
\[
    1<\delta<\frac32,
\]
there exists $\eps=\eps(n,\delta)>0$ for which the implication holds.
\end{theorem}

\subsection{Preliminary identities}

For every $z\in\R^n$ and fixed $m$,
\begin{equation}\label{eq:app-A-formula}
        z^TA^{(m)}z =2\sum_{i\ne m}\sigma_1(\lambda|im)z_i^2 -2\sum_{i<j}\sigma_1(\lambda|ij)z_iz_j.
\end{equation}
Moreover,
\begin{equation}\label{eq:app-D2-sigma3}
        D^2\sigma_3(\lambda)[z,z]
        =2\sum_{i<j}\sigma_1(\lambda|ij)z_iz_j,
\end{equation}
so
\begin{equation}\label{eq:app-A-D2}
        z^TA^{(m)}z
        =-D^2\sigma_3(\lambda)[z,z]
        +2\sum_{i\ne m}\sigma_1(\lambda|im)z_i^2.
\end{equation}
The concavity of $\sigma_3^{1/3}$ on $\Gamma_3$ gives
\begin{equation}\label{eq:app-hodge}
        D^2\sigma_3(\lambda)[z,z]
        \le \frac23\frac{(d(\lambda)^Tz)^2}{\sigma_3(\lambda)}.
\end{equation}
Thus $-D^2\sigma_3(\lambda)[z,z]\ge0$ whenever $d^Tz=0$.  Consequently, on the constraint plane $d^Tz=0$,
\begin{equation}\label{eq:app-coercive}
        z^TA^{(m)}z
        \ge2\sum_{i\ne m}\sigma_1(\lambda|im)z_i^2.
\end{equation}

\subsection{Compactness reduction and structure of the limit}

Both $\Theta$ and the condition $d_1\le\eps\sigma_2(\lambda)$ are homogeneous of degree zero.  Since the conclusion is independent of the ordering of the coordinates after relabeling, we assume without loss of generality that
\[
        \lambda_1\ge\lambda_2\ge\cdots\ge\lambda_n.
\]
It is enough to prove the following compactness assertion.

\begin{proposition}\label{prop:app-compactness}
Let $\lambda^{(\ell)}\in\Gamma_3$ be ordered and satisfy
\[
        \sigma_1(\lambda^{(\ell)})=1,
        \qquad  \frac{d_1^{(\ell)}}{\sigma_2(\lambda^{(\ell)})}\to0.
\]
Then
\[
    \liminf_{\ell\to\infty}\Theta(\lambda^{(\ell)})\ge\frac32.
\]
\end{proposition}

Assume the hypotheses of Proposition \ref{prop:app-compactness}.  Since
\[
    \sigma_2(\lambda)=\frac12(\sigma_1(\lambda)^2-|\lambda|^2)>0,
\]
the normalization $\sigma_1(\lambda)=1$ gives $|\lambda|<1$.  After passing to a subsequence,
\[
    \lambda^{(\ell)}\to\lambda\in\ol\Gamma_3,
        \qquad
        \sigma_1(\lambda)=1.
\]
Because $\sigma_2(\lambda^{(\ell)})\le1/2$, the hypothesis implies
\begin{equation}\label{eq:app-sigma2-eta-zero}
        \sigma_2(\lambda|1)=0.
\end{equation}
Let
\[
    \eta=(\lambda_2,\ldots,\lambda_n),
        \qquad
        S=\sigma_1(\eta)=1-\lambda_1.
\]
Then \eqref{eq:app-sigma2-eta-zero} becomes
\begin{equation}\label{eq:app-sigma2-eta}
        \sigma_2(\eta)=0.
\end{equation}
Since $\lambda\in\ol\Gamma_3$,
\[
        \sigma_2(\lambda)=\lambda_1S\ge0,
        \qquad
        \sigma_3(\lambda)=\sigma_3(\eta)\ge0.
\]
As $\lambda_1\ge1/n>0$, we have $S\ge0$.

If $S=0$, then \eqref{eq:app-sigma2-eta} gives
\[
        |\eta|^2=S^2-2\sigma_2(\eta)=0,
\]
hence $\eta=0$.  If $S>0$, put $\zeta=\eta/S$.  Then
\[
        \sum_i\zeta_i=1,
        \qquad
        \sigma_2(\zeta)=0,
        \qquad
        \sum_i\zeta_i^2=1.
\]
In particular $\zeta_i\le1$ for every $i$, and therefore
\[
        1-\sum_i\zeta_i^3
        =\sum_i\zeta_i^2(1-\zeta_i)\ge0.
\]
Newton's identity gives
\[
        3\sigma_3(\zeta)
        =\sum_i\zeta_i^3+3\sigma_1(\zeta)\sigma_2(\zeta)-\sigma_1(\zeta)^3
        =\sum_i\zeta_i^3-1.
\]
Thus $\sigma_3(\zeta)\le0$.  On the other hand,
\[
        \sigma_3(\eta)=S^3\sigma_3(\zeta)\ge0.
\]
Hence $\sigma_3(\zeta)=0$, and
\[
        \zeta_i^2(1-\zeta_i)=0
        \quad\text{for every }i.
\]
Thus each $\zeta_i$ is either $0$ or $1$, and because $\sum_i\zeta_i=1$, exactly one $\zeta_i$ equals $1$.

Consequently, the limit of $\lambda^{(\ell)}$ has one of the two forms
\begin{equation}\label{eq:app-rank-one-limit}
        \lambda=(1,0,\ldots,0),
\end{equation}
or
\begin{equation}\label{eq:app-rank-two-limit}
        \lambda=(a,b,0,\ldots,0),
        \qquad
        a\ge b>0,
        \qquad
        a+b=1.
\end{equation}
We call these the rank-one and rank-two cases and treat them separately.

\subsection{The rank-one limit}

Assume first that \eqref{eq:app-rank-one-limit} holds.  Since $\lambda_1^{(\ell)}\to1$, homogeneity allows us to divide by $\lambda_1^{(\ell)}$ and write
\begin{equation}\label{eq:app-rank-one-param}
        \lambda^{(\ell)}=(1,\rho_\ell\xi^{(\ell)}),
        \qquad
        \rho_\ell=\frac{\lambda_2^{(\ell)}}{\lambda_1^{(\ell)}}\to0^+,
        \qquad
        \xi_1^{(\ell)}=1.
\end{equation}
Here $\xi^{(\ell)}\in\R^N$, where $N=n-1$.  Throughout this subsection, $e$ denotes the all-ones vector in $\R^N$ whenever it is applied to the small block.  Since $\xi^{(\ell)}\in\Gamma_2$ and $\xi_1^{(\ell)}=1$, the sequence $\xi^{(\ell)}$ is bounded.  Passing to a subsequence,
\begin{equation}\label{eq:app-xi-limit}
        \xi^{(\ell)}\to\xi\in\ol\Gamma_2,
        \qquad
        \sigma_1(\xi)>0.
\end{equation}

\subsubsection{The index \texorpdfstring{$m=1$}{}}

For $\lambda=(1,\rho\xi)$, set
\[
        S=\sigma_1(\xi),
        \qquad
        s=\sigma_2(\xi),
        \qquad
        t=\sigma_3(\xi),
\]
and
\[
        b:=D\sigma_2(\xi)\in\R^N,
        \qquad
        c:=D\sigma_3(\xi)\in\R^N.
\]
Thus
\[
        b_i=S-\xi_i=\sigma_1(\xi|i),
        \qquad
        c_i=\sigma_2(\xi|i).
\]
Write a feasible vector as $z=(z_0,w)$, $w\in\R^N$.  Then
\[
        d_1=\rho^2s,
        \qquad
        d_{i+1}=\rho b_i+\rho^2c_i
        \quad(i=1,\ldots,N).
\]
The two constraints are
\begin{equation}\label{eq:app-r1-constraint-d}
        (b+\rho c)\cdot w+\rho s z_0=0,
\end{equation}
\begin{equation}\label{eq:app-r1-constraint-e}
        z_0+e\cdot w=1.
\end{equation}
For fixed $\ell$, since $\xi^{(\ell)}\in\Gamma_2$, we have
\[
        b_{\ell,i}=\sigma_1(\xi^{(\ell)}|i)>0
        \quad(i=1,\ldots,N).
\]
Hence, for $z=(z_0,w)$ in the feasible set, \eqref{eq:app-coercive} gives
\[
        z^TA^{(1)}(\lambda^{(\ell)})z
        \ge2\rho_\ell\sum_{i=1}^N b_{\ell,i}w_i^2
        \ge2\rho_\ell b_{\ell,*}|w|^2,
        \qquad b_{\ell,*}:=\min_i b_{\ell,i}>0.
\]
Together with $z_0+e\cdot w=1$, this implies coercivity on the feasible set.  Thus the infimum $I_1(\lambda^{(\ell)})$ is attained.  Take minimizers
\[
        z^{(\ell)}=(z_0^{(\ell)},w^{(\ell)}).
\]
Put
\begin{equation}\label{eq:app-R-ell-def}
        R_\ell:=\frac{(z^{(\ell)})^TA^{(1)}(\lambda^{(\ell)})z^{(\ell)}}{\rho_\ell^2s_\ell},
        \qquad
        s_\ell:=\sigma_2(\xi^{(\ell)}).
\end{equation}
If $R_\ell$ is unbounded along a subsequence, then the desired lower bound for $\Theta_1$ is immediate along that subsequence.  Thus, after passing to a subsequence, we assume $R_\ell\le C$.  Let
\begin{equation}\label{eq:app-tau-def}
        \tau_\ell:=\frac{\rho_\ell t_\ell}{s_\ell},
        \qquad
        t_\ell:=\sigma_3(\xi^{(\ell)}).
\end{equation}
Since
\[
        \sigma_3(\lambda^{(\ell)})=\rho_\ell^2(s_\ell+\rho_\ell t_\ell)>0,
\]
we have $\tau_\ell>-1$.  If $\sigma_2(\xi)>0$, then $s_\ell\to s_0>0$, while $t_\ell$ remains bounded, so $\tau_\ell\to0$.  If $\sigma_2(\xi)=0$, then $\sigma_3(\xi)\le0$.  Maclaurin's inequality gives $t_\ell\le Cs_\ell^{3/2}$, hence
\[
        \tau_\ell=\rho_\ell\frac{t_\ell}{s_\ell}
        \le C\rho_\ell s_\ell^{1/2}\to0.
\]
Thus, in both cases, after passing to a subsequence,
\begin{equation}\label{eq:app-tau-limit}
        \tau_\ell\to\tau\in[-1,0].
\end{equation}
For the optimization problem
\[
        \inf_{d^Tz=0,\ e^Tz=1} z^TA^{(1)}z,
\]
the KKT equation is
\begin{equation}\label{eq:app-KKT-r1}
        2A^{(1)}z^{(\ell)}+\alpha_\ell d^{(\ell)}+\beta_\ell e=0.
\end{equation}
Multiplying by $z^{(\ell)}$ and using \eqref{eq:app-R-ell-def} and the two constraints gives
\begin{equation}\label{eq:app-beta-est}
        \beta_\ell=-2(z^{(\ell)})^TA^{(1)}z^{(\ell)}
        =-2\rho_\ell^2s_\ell R_\ell
        =O(\rho_\ell^2s_\ell).
\end{equation}
Write
\[
        M_\ell=M_0+\rho_\ell M_1^\ell\in\R^{N\times N},
        \qquad
        M_0=I-ee^T,
\]
where
\[
        (M_1^\ell)_{ii}=2(b_\ell)_i,
        \qquad
        (M_1^\ell)_{ij}=-(S_\ell-\xi_i^{(\ell)}-\xi_j^{(\ell)})
        \quad(i\ne j).
\]
Since $N\ge2$, $M_0$ is invertible and
\begin{equation}\label{eq:app-P-def}
        P:=M_0^{-1}=I-\frac{1}{N-1}ee^T.
\end{equation}
Define
\[
        U_\ell:=z_0^{(\ell)}-\frac{\alpha_\ell}{2}.
\]
Then the $w^{(\ell)}$ part of \eqref{eq:app-KKT-r1} is
\[
        M_\ell w^{(\ell)}
        =\rho_\ell U_\ell b_\ell
        -\frac{\alpha_\ell\rho_\ell^2}{2}c_\ell
        -\frac{\beta_\ell}{2}e.
\]
Since $M_0$ is invertible and $\rho_\ell\to0$, $M_\ell$ is also invertible for large $\ell$.  Thus
\begin{equation}\label{eq:app-w-expression}
        w^{(\ell)}
        =\rho_\ell U_\ell M_\ell^{-1}b_\ell
        -\frac{\alpha_\ell\rho_\ell^2}{2}M_\ell^{-1}c_\ell
        -\frac{\beta_\ell}{2}M_\ell^{-1}e.
\end{equation}
Substitute \eqref{eq:app-w-expression} into the constraint \eqref{eq:app-r1-constraint-d}.  Omitting the index $\ell$ temporarily,
\begin{align}\label{eq:app-constraint-expanded}
        0={}&\rho sz_0+\rho U b^TM^{-1}b+\rho^2U c^TM^{-1}b
        -\frac{\alpha\rho^2}{2}b^TM^{-1}c\notag\\
        &-\frac{\alpha\rho^3}{2}c^TM^{-1}c+O(|\beta|).
\end{align}
The normalization constraint \eqref{eq:app-r1-constraint-e} and \eqref{eq:app-w-expression} also give
\begin{equation}\label{eq:app-z0-est}
        z_0^{(\ell)}=1+O\bigl(\rho_\ell(1+|U_\ell|)\bigr).
\end{equation}
Indeed, $z_0=1-e\cdot w$ and \eqref{eq:app-w-expression} imply
\[
        |z_0-1|=|e\cdot w|
        \le C\rho|U|+C\rho^2|\alpha|+C\rho^2,
\]
where boundedness of $M_\ell^{-1}$ is used.  Since $\alpha=2(z_0-U)$,
\[
        |\alpha|\le2|z_0-1|+2+2|U|.
\]
Thus $|e\cdot w|\le C\rho+C\rho|U|$, after absorbing the $C\rho^2|z_0-1|$ term into the left-hand side.  This proves \eqref{eq:app-z0-est}.  Also
\begin{equation}\label{eq:app-alpha-est}
        |\alpha|\le C+C|U|.
\end{equation}
Now we estimate the other terms in \eqref{eq:app-constraint-expanded}.  Direct computation gives
\begin{equation}\label{eq:app-P-identities}
        Pb=-\xi,
        \qquad
        b^TPb=-2s,
        \qquad
        c\cdot\xi=3t.
\end{equation}
Also,
\[
        (M_1\xi)_i=4\xi_i(S-\xi_i)-2s,
\]
and therefore
\begin{align}\label{eq:app-xiMxi}
        \xi^TM_1\xi
        &=\sum_{i=1}^N\xi_i(M_1\xi)_i \ = \sum_{i=1}^N\bigl(4\xi_i(S-\xi_i)-2s\bigr)\xi_i\notag\\
        &=4S\sum_{i=1}^N\xi_i^2-4\sum_{i=1}^N\xi_i^3-2Ss \notag \\
        & =4S(S^2-2s)-4(S^3-3Ss+3t)-2Ss \notag \\
        &=2Ss-12t.
\end{align}
The inverse matrix $M^{-1}$ can be expressed as
\[
        M^{-1}=P-\rho PM_1M^{-1},
        \qquad
        M^{-1}=P-\rho M^{-1}M_1P.
\]
Substituting the latter equality into the former gives
\begin{equation}\label{eq:app-M-inverse-expansion}
        M^{-1}=P-\rho PM_1P+\rho^2PM_1M^{-1}M_1P.
\end{equation}
Combining \eqref{eq:app-P-identities} and \eqref{eq:app-xiMxi}, we have
\begin{align}\label{eq:app-bMinvb}
        b^TM^{-1}b
        &=b^TPb-\rho(Pb)^TM_1(Pb)+\rho^2(M_1Pb)^TM^{-1}(M_1Pb)\notag\\
        &=-2s-\rho(2Ss-12t)+\rho^2O(|M_1Pb|^2).
\end{align}
Since $M_1Pb=-M_1\xi$ and $(M_1\xi)_i=4\xi_i(S-\xi_i)-2s$,
\[
        |M_1\xi|^2
        \le C\sum_i\xi_i^2(S-\xi_i)^2+Cs
        \le C\sum_i\xi_i^2(S-\xi_i)+Cs.
\]
Using
\[
        \sum_i\xi_i^2(S-\xi_i)=Ss-3t,
\]
we get
\begin{equation}\label{eq:app-M1xi-bound}
        |M_1\xi|^2\le C(s+|t|).
\end{equation}
Thus
\begin{equation}\label{eq:app-bMinvb-final}
        b^TM^{-1}b
        =-2s-\rho(2Ss-12t)+O(\rho^2(s+|t|)).
\end{equation}
Similarly, from \eqref{eq:app-M-inverse-expansion},
\begin{align*}
        b^TM^{-1}c
        &=b^TPc-\rho(Pb)^TM_1(Pc)+\rho^2 b^TPM_1M^{-1}M_1Pc\notag\\
        &=-\xi\cdot c+\rho(M_1\xi)^T(Pc)+\rho^2(M_1Pb)^TM^{-1}M_1(Pc).
\end{align*}
Direct computation shows that
\[
        (Pc)_i=\frac{s}{N-1}-\xi_i(S-\xi_i),
\]
which implies $|Pc|^2=O(s+|t|)$.  Combining this with \eqref{eq:app-P-identities} and \eqref{eq:app-M1xi-bound}, we have
\begin{equation}\label{eq:app-bMinvc}
        b^TM^{-1}c=-3t+O(\rho(s+|t|)).
\end{equation}
Finally,
\[
        |c|^2
        =\sum_i\bigl(s-\xi_i(S-\xi_i)\bigr)^2
        \le Cs+\sum_i\xi_i^2(S-\xi_i)
        =Cs+Ss-3t
        \le C(s+|t|),
\]
and the boundedness of $M^{-1}$ gives
\begin{equation}\label{eq:app-cMinvc}
        c^TM^{-1}c=O(s+|t|).
\end{equation}
Dividing \eqref{eq:app-constraint-expanded} by $\rho s$ and using \eqref{eq:app-beta-est}, \eqref{eq:app-alpha-est}, and \eqref{eq:app-bMinvb-final}--\eqref{eq:app-cMinvc}, we obtain
\begin{align*}
        0={}&z_0^{(\ell)}
        +U_\ell\frac{b_\ell^TM_\ell^{-1}b_\ell}{s_\ell}
        +\rho_\ell U_\ell\frac{b_\ell^TM_\ell^{-1}c_\ell}{s_\ell}
        -\frac{\alpha_\ell\rho_\ell}{2}\frac{b_\ell^TM_\ell^{-1}c_\ell}{s_\ell}
        -\frac{\alpha_\ell\rho_\ell^2}{2}\frac{c_\ell^TM_\ell^{-1}c_\ell}{s_\ell}
        +O(\rho_\ell)\\
        ={}&z_0^{(\ell)}+
        \left(-2+9\frac{\rho_\ell t_\ell}{s_\ell}\right)U_\ell
        +\frac32\frac{\rho_\ell t_\ell}{s_\ell}\alpha_\ell
        +o(1)(1+|U_\ell|)\\
        ={}&z_0^{(\ell)}+(-2+9\tau_\ell)U_\ell+\frac32\tau_\ell\alpha_\ell
        +o(1)(1+|U_\ell|).
\end{align*}
Since $\alpha_\ell/2=z_0^{(\ell)}-U_\ell$, we get
\begin{equation}\label{eq:app-U-equation}
        (1+3\tau_\ell)z_0^{(\ell)}+(-2+6\tau_\ell)U_\ell
        =o(1)(1+|U_\ell|).
\end{equation}
By \eqref{eq:app-tau-limit}, the coefficient $-2+6\tau_\ell$ is bounded away from zero.  Combining \eqref{eq:app-z0-est} and \eqref{eq:app-U-equation} gives $U_\ell=O(1)$.  Then \eqref{eq:app-w-expression} yields $w^{(\ell)}=O(\rho_\ell)$.  After passing to a subsequence, $U_\ell\to U$.  Letting $\ell\to\infty$ in \eqref{eq:app-U-equation} and using $z_0^{(\ell)}\to1$ gives
\begin{equation}\label{eq:app-U-limit}
        U=\frac{1+3\tau}{2-6\tau}.
\end{equation}
It remains to compute the limiting quotient.  The first component of \eqref{eq:app-KKT-r1} is
\[
        2(A^{(1)}z^{(\ell)})_0+
        \alpha_\ell\rho_\ell^2s_\ell+
        \beta_\ell=0.
\]
Since $(A^{(1)}z)_0=-\rho b\cdot w$ and \eqref{eq:app-r1-constraint-d} gives
\[
        -b\cdot w=\rho sz_0+\rho c\cdot w,
\]
we obtain from \eqref{eq:app-beta-est} that
\begin{equation}\label{eq:app-R-expression}
        R_\ell=z_0^{(\ell)}+\frac{c_\ell\cdot w^{(\ell)}}{s_\ell}+\frac{\alpha_\ell}{2}.
\end{equation}
Using \eqref{eq:app-w-expression}, \eqref{eq:app-bMinvc}, and \eqref{eq:app-cMinvc},
\[
        \frac{c_\ell\cdot w^{(\ell)}}{s_\ell}
        =-3\tau_\ell U_\ell+o(1).
\]
Together with $\alpha_\ell/2=z_0^{(\ell)}-U_\ell$, \eqref{eq:app-R-expression} becomes
\[
        R_\ell=2z_0^{(\ell)}-(1+3\tau_\ell)U_\ell+o(1).
\]
Therefore, by \eqref{eq:app-U-limit},
\[
        \lim_{\ell\to\infty}R_\ell
        =2-\frac{(1+3\tau)^2}{2-6\tau}.
\]
For $\tau\in[-1,0]$,
\[
        2-\frac{(1+3\tau)^2}{2-6\tau}-\frac32
        =\frac{9\tau(\tau+1)}{2(3\tau-1)}\ge0.
\]
Thus $\liminf_{\ell\to\infty}R_\ell\ge3/2$ and
\[
        \Theta_1(\lambda^{(\ell)})
        =\frac{\sigma_1(\lambda^{(\ell)})}{d_1^{(\ell)}}I_1(\lambda^{(\ell)})
        =(1+\rho_\ell S_\ell)R_\ell,
\]
where $1+\rho_\ell S_\ell\to1$.  Hence
\begin{equation}\label{eq:app-theta1-rankone}
        \liminf_{\ell\to\infty}\Theta_1(\lambda^{(\ell)})\ge\frac32.
\end{equation}

\subsubsection{The indices \texorpdfstring{$m\ge2$}{} in the rank-one limit}

Fix $m\ge2$.  Write
\[
        \lambda_1=1,
        \qquad
        \lambda_m=\rho a,
        \qquad
        \lambda_j=\rho\eta_j\quad(j\in J),
\]
where $J=\{2,\ldots,n\}\setminus\{m\}$.  Set
\[
        T=\sigma_1(\eta),
        \qquad
        H=\sigma_2(\eta),
        \qquad
        K=\sigma_3(\eta).
\]
Then
\begin{equation}\label{eq:app-d1-dm-rankone}
        d_1=\rho^2(aT+H),
        \qquad
        d_m=\rho T+\rho^2H=\rho(T+\rho H),
\end{equation}
\begin{equation}\label{eq:app-dj-rankone}
        d_j=\rho(a+T-\eta_j)+\rho^2\bigl(a(T-\eta_j)+\sigma_2(\eta|j)\bigr)
        \quad(j\in J).
\end{equation}
Moreover,
\begin{equation}\label{eq:app-sigma1-pairs-rankone}
        \sigma_1(\lambda|1m)=\rho T,
        \qquad
        \sigma_1(\lambda|jm)=1+\rho(T-\eta_j)
        \quad(j\in J).
\end{equation}
Notice that $\lambda\in\Gamma_3$ gives $T>0$.

It suffices to consider sequences of feasible vectors with bounded quotient; otherwise the desired lower bound is immediate.  Thus let
\[
        z^{(\ell)}=(z_1^{(\ell)},z_m^{(\ell)},u^{(\ell)}),
        \qquad
        u^{(\ell)}=(z_j^{(\ell)})_{j\in J},
\]
with
\begin{equation}\label{eq:app-q-bounded-rankone}
        q_\ell:=(z^{(\ell)})^TA^{(m)}(\lambda^{(\ell)})z^{(\ell)}
        \le C d_m^{(\ell)}.
\end{equation}
By \eqref{eq:app-coercive} and \eqref{eq:app-sigma1-pairs-rankone},
\[
        q_\ell\ge2\rho_\ell T_\ell(z_1^{(\ell)})^2+(2+o(1))|u^{(\ell)}|^2.
\]
Hence
\begin{equation}\label{eq:app-z1-u-bound-rankone}
        \rho_\ell T_\ell(z_1^{(\ell)})^2+|u^{(\ell)}|^2\le C d_m^{(\ell)}.
\end{equation}
If $T_\ell\to T_0>0$, then $d_m^{(\ell)}=\rho_\ell T_\ell(1+o(1))$.  From \eqref{eq:app-z1-u-bound-rankone}, $z_1^{(\ell)}=O(1)$ and $u^{(\ell)}\to0$.  Dividing $d^Tz=0$ by $d_m^{(\ell)}$ gives
\[
        z_m^{(\ell)}
        =-\frac{d_1^{(\ell)}}{d_m^{(\ell)}}z_1^{(\ell)}
        -\sum_{j\in J}\frac{d_j^{(\ell)}}{d_m^{(\ell)}}u_j^{(\ell)}\to0,
\]
because $d_1^{(\ell)}/d_m^{(\ell)}=O(\rho_\ell)$ and $d_j^{(\ell)}/d_m^{(\ell)}=O(1)$.  The normalization $e^Tz=1$ then gives $z_1^{(\ell)}\to1$.

It remains to treat the case $T_\ell\to0$.  Since $\sigma_1(\xi^{(\ell)})=a_\ell+T_\ell$ has a positive limit by \eqref{eq:app-xi-limit}, we must have
\begin{equation}\label{eq:app-a-positive}
        a_\ell\to a_0>0.
\end{equation}
Also $\eta^{(\ell)}\to0$: indeed, $\sigma_1(\eta)=0$ and $\xi=(a_0,\eta)\in\ol\Gamma_2$, so
\[
        0\le \sigma_2(\xi)=a_0\sigma_1(\eta)+\sigma_2(\eta)
        =\sigma_2(\eta)=-\frac12|\eta|^2,
\]
which forces $\eta=0$.  We first claim that
\begin{equation}\label{eq:app-H-bound}
        |H_\ell|\le C T_\ell.
\end{equation}
If $H_\ell\ge0$, \eqref{eq:app-H-bound} follows from
\[
        2H_\ell=T_\ell^2-|\eta^{(\ell)}|^2\le T_\ell^2.
\]
If $H_\ell<0$, suppose for contradiction that $-H_\ell/T_\ell\to+\infty$.  Then
\[
        \sigma_3(\lambda^{(\ell)})
        =\rho_\ell^2(a_\ell T_\ell+H_\ell)+\rho_\ell^3(a_\ell H_\ell+K_\ell)>0
\]
gives
\[
        a_\ell T_\ell+H_\ell+\rho_\ell(a_\ell H_\ell+K_\ell)>0.
\]
But
\[
        |\eta^{(\ell)}|^2=T_\ell^2-2H_\ell\le C|H_\ell|,
\]
and
\[
        |K_\ell|\le C|\eta^{(\ell)}|^3\le C|\eta^{(\ell)}|^2\le C|H_\ell|.
\]
This forces
\[
        a_\ell T_\ell+H_\ell+
\rho_\ell(a_\ell H_\ell+K_\ell)
        =a_\ell T_\ell-|H_\ell|+o(|H_\ell|)<0
\]
for large $\ell$, a contradiction.  This proves \eqref{eq:app-H-bound}.  Consequently,
\begin{equation}\label{eq:app-dm-asym-rankone}
        d_m^{(\ell)}=\rho_\ell T_\ell(1+O(\rho_\ell)).
\end{equation}
Using \eqref{eq:app-z1-u-bound-rankone},
\begin{equation}\label{eq:app-z1-u-bound-rankone2}
        z_1^{(\ell)}=O(1),
        \qquad
        |u^{(\ell)}|^2=O(\rho_\ell T_\ell),
\end{equation}
and then $z_m^{(\ell)}=O(1)$ by $e^Tz=1$.  Let
\[
        r_\ell:=\sum_{j\in J}u_j^{(\ell)}.
\]
Expanding $d^Tz=0$ with \eqref{eq:app-d1-dm-rankone}--\eqref{eq:app-dj-rankone} and dividing by $\rho_\ell$ gives
\begin{align}\label{eq:app-dconstraint-rankone-expanded}
        0={}&\rho_\ell(a_\ell T_\ell+H_\ell)z_1^{(\ell)}+(T_\ell+\rho_\ell H_\ell)z_m^{(\ell)}\notag\\
        &+\sum_{j\in J}(a_\ell+T_\ell-\eta_j^{(\ell)})u_j^{(\ell)}\notag\\
        &+\rho_\ell\sum_{j\in J}
        \bigl(a_\ell(T_\ell-\eta_j^{(\ell)})+\sigma_2(\eta^{(\ell)}|j)\bigr)u_j^{(\ell)}.
\end{align}
Since \eqref{eq:app-H-bound} implies
\[
        |\eta^{(\ell)}|^2=T_\ell^2-2H_\ell=O(T_\ell),
\]
while \eqref{eq:app-z1-u-bound-rankone2} gives $|u^{(\ell)}|=O((\rho_\ell T_\ell)^{1/2})$, we have
\begin{equation}\label{eq:app-small-term1}
        \left|\sum_{j\in J}(T_\ell-\eta_j^{(\ell)})u_j^{(\ell)}\right|
        \le C|\eta^{(\ell)}|\,|u^{(\ell)}|
        \le C T_\ell^{1/2}\rho_\ell^{1/2}T_\ell^{1/2}=o(T_\ell),
\end{equation}
and similarly
\begin{equation}\label{eq:app-small-term2}
        \left|\sum_{j\in J}\sigma_2(\eta^{(\ell)}|j)u_j^{(\ell)}\right|
        \le C T_\ell^{1/2}\rho_\ell^{1/2}T_\ell^{1/2}=o(T_\ell).
\end{equation}
Thus all terms in \eqref{eq:app-dconstraint-rankone-expanded} except $a_\ell r_\ell$ and $T_\ell z_m^{(\ell)}$ are $o(T_\ell)$.  Hence
\[
        a_\ell r_\ell=-T_\ell z_m^{(\ell)}+o(T_\ell),
\]
and \eqref{eq:app-a-positive} gives
\begin{equation}\label{eq:app-r-ell-relation}
        r_\ell=-\frac{T_\ell z_m^{(\ell)}}{a_\ell}+o(T_\ell).
\end{equation}
A direct expansion of $q_\ell$ gives
\begin{align*}
        q_\ell={}&2\rho_\ell T_\ell(z_1^{(\ell)})^2
        +2\sum_{j\in J}(1+\rho_\ell(T_\ell-\eta_j^{(\ell)}))(u_j^{(\ell)})^2
        -2\rho_\ell T_\ell z_1^{(\ell)}z_m^{(\ell)}\\
        &-2\rho_\ell\sum_{j\in J}(a_\ell+T_\ell-\eta_j^{(\ell)})z_1^{(\ell)}u_j^{(\ell)}
        -2\sum_{j\in J}(1+\rho_\ell(T_\ell-\eta_j^{(\ell)}))z_m^{(\ell)}u_j^{(\ell)}\\
        &-2\sum_{\substack{j<k\\ j,k\in J}}
        (1+\rho_\ell(a_\ell+T_\ell-\eta_j^{(\ell)}-\eta_k^{(\ell)}))u_j^{(\ell)}u_k^{(\ell)}.
\end{align*}
Using \eqref{eq:app-z1-u-bound-rankone2}, \eqref{eq:app-small-term1}, and \eqref{eq:app-r-ell-relation}, we get
\begin{equation}\label{eq:app-q-rankone-asym}
        q_\ell=-2z_m^{(\ell)}r_\ell+o(T_\ell)
        =\frac{2T_\ell}{a_\ell}(z_m^{(\ell)})^2+o(T_\ell).
\end{equation}
If $z_m^{(\ell)}\not\to0$, then \eqref{eq:app-dm-asym-rankone} and \eqref{eq:app-q-rankone-asym} imply $q_\ell/d_m^{(\ell)}\to+\infty$, contradicting \eqref{eq:app-q-bounded-rankone}.  Therefore $z_m^{(\ell)}\to0$, and then $z_1^{(\ell)}\to1$ by $e^Tz=1$.

Thus, in both cases $T_\ell\to T_0>0$ and $T_\ell\to0$, we have proved
\[
        z_1^{(\ell)}\to1,
        \qquad
        z_m^{(\ell)}\to0,
        \qquad
        u^{(\ell)}\to0.
\]
Returning to \eqref{eq:app-coercive},
\[
        q_\ell\ge2\rho_\ell T_\ell(z_1^{(\ell)})^2.
\]
Since $d_m^{(\ell)}=\rho_\ell T_\ell(1+o(1))$ in both cases, we obtain
\[
        \liminf_{\ell\to\infty}\frac{q_\ell}{d_m^{(\ell)}}\ge2.
\]
Recalling that $\sigma_1(\lambda^{(\ell)})=1+\rho_\ell\sigma_1(\xi^{(\ell)})\to1$, we have
\begin{equation}\label{eq:app-thetam-rankone}
        \liminf_{\ell\to\infty}\Theta_m(\lambda^{(\ell)})\ge2,
        \qquad m\ge2.
\end{equation}
Combining \eqref{eq:app-theta1-rankone} and \eqref{eq:app-thetam-rankone},
\begin{equation}\label{eq:app-theta-rankone-final}
        \liminf_{\ell\to\infty}\Theta(\lambda^{(\ell)})\ge\frac32
\end{equation}
whenever the limiting vector is \eqref{eq:app-rank-one-limit}.

\subsection{The rank-two limit}

Assume now that
\[
        \lambda=(a,b,0,\ldots,0),
        \qquad
        a\ge b>0,
        \qquad
        a+b=1.
\]
Write
\[
        \lambda^{(\ell)}=(a_\ell,b_\ell,\xi^{(\ell)}),
        \qquad
        a_\ell\to a,
        \qquad
        b_\ell\to b,
        \qquad
        \xi^{(\ell)}\to0,
\]
where $\xi^{(\ell)}=(\lambda_3^{(\ell)},\ldots,\lambda_n^{(\ell)})$.  We normalize by $\sigma_1(\lambda^{(\ell)})=1$, so
\[
        a_\ell+b_\ell+T_\ell=1,
        \qquad
        T_\ell:=\sigma_1(\xi^{(\ell)}).
\]
Set
\[
        H_\ell:=\sigma_2(\xi^{(\ell)}),
        \qquad
        K_\ell:=\sigma_3(\xi^{(\ell)}).
\]
Since $\xi^{(\ell)}\in\Gamma_1$ and $\xi^{(\ell)}\to0$, we have $T_\ell>0$ and $T_\ell\to0$.  Also
\[
        d_1^{(\ell)}=b_\ell T_\ell+H_\ell,
        \qquad
        d_2^{(\ell)}=a_\ell T_\ell+H_\ell,
\]
and, for $j\ge3$,
\begin{equation}\label{eq:app-dj-ranktwo}
        d_j^{(\ell)}=a_\ell b_\ell+O(|\xi^{(\ell)}|)\to ab.
\end{equation}
Since
\begin{align}\label{eq:app-sigma3-ranktwo}
        0<\sigma_3(\lambda^{(\ell)})
        &=a_\ell b_\ell T_\ell+(a_\ell+b_\ell)H_\ell+K_\ell\notag\\
        &\le a_\ell b_\ell T_\ell+\frac{a_\ell+b_\ell}{2}(T_\ell^2-|\xi^{(\ell)}|^2)+C|\xi^{(\ell)}|^3,
\end{align}
combining $\xi^{(\ell)}\to0$ and $T_\ell\to0$, we get
\begin{equation}\label{eq:app-xi-square-bound}
        |\xi^{(\ell)}|^2\le C T_\ell.
\end{equation}
Put
\[
        p^{(\ell)}:=\frac{\xi^{(\ell)}}{T_\ell^{1/2}},
        \qquad
        h_\ell:=\frac{|\xi^{(\ell)}|^2}{2T_\ell}.
\]
After passing to a subsequence,
\[
        p^{(\ell)}\to p,
        \qquad
        h_\ell\to h\ge0.
\]
Moreover,
\begin{equation}\label{eq:app-p-h-identities}
        \sum_{j\ge3}p_j=0,
        \qquad
        |p|^2=2h,
        \qquad
        \frac{H_\ell}{T_\ell}\to -h.
\end{equation}
Dividing \eqref{eq:app-sigma3-ranktwo} by $T_\ell$ and using \eqref{eq:app-xi-square-bound} gives
\begin{equation}\label{eq:app-h-bound-ranktwo}
        0\le h\le ab.
\end{equation}
Consequently,
\begin{equation}\label{eq:app-d1-d2-limits-ranktwo}
        \frac{d_1^{(\ell)}}{T_\ell}\to b-h,
        \qquad
        \frac{d_2^{(\ell)}}{T_\ell}\to a-h.
\end{equation}
Because $h\le ab$, both limits are strictly positive:
\[
        b-h\ge b^2>0,
        \qquad
        a-h\ge a^2>0.
\]
Thus
\begin{equation}\label{eq:app-d12-comparable-T}
        cT_\ell\le d_1^{(\ell)}\le CT_\ell,
        \qquad
        cT_\ell\le d_2^{(\ell)}\le CT_\ell.
\end{equation}

\subsubsection{The positive indices \texorpdfstring{$m=1$}{} and \texorpdfstring{$m=2$}{}}

First consider $m=1$.  Write
\[
        z^{(\ell)}=(x_\ell,y_\ell,u^{(\ell)}),
        \qquad
        x_\ell=z_1^{(\ell)},
        \qquad
        y_\ell=z_2^{(\ell)},
        \qquad
        u^{(\ell)}=(z_3^{(\ell)},\ldots,z_n^{(\ell)}).
\]
Assume
\begin{equation}\label{eq:app-ranktwo-bounded-quotient}
        \frac{(z^{(\ell)})^TA^{(1)}(\lambda^{(\ell)})z^{(\ell)}}{d_1^{(\ell)}}\le C.
\end{equation}
By \eqref{eq:app-coercive} and \eqref{eq:app-d12-comparable-T},
\[
        (z^{(\ell)})^TA^{(1)}z^{(\ell)}
        \ge2T_\ell y_\ell^2
        +2\sum_{j\ge3}(b_\ell+T_\ell-\xi_j^{(\ell)})(u_j^{(\ell)})^2.
\]
The coefficients $b_\ell+T_\ell-\xi_j^{(\ell)}$ are bounded below by $b/2$ for large $\ell$.  Hence
\[
        y_\ell=O(1),
        \qquad
        u^{(\ell)}=O(T_\ell^{1/2}).
\]
The constraint $e^Tz=1$ gives $x_\ell=O(1)$.  Set
\[
        v^{(\ell)}:=\frac{u^{(\ell)}}{T_\ell^{1/2}}.
\]
Passing to a subsequence,
\[
        x_\ell\to x,
        \qquad
        y_\ell\to y,
        \qquad
        v^{(\ell)}\to v.
\]
From $e^Tz=1$,
\begin{equation}\label{eq:app-x-y-sum}
        x+y=1.
\end{equation}
For $j\ge3$,
\[
        d_j^{(\ell)}=a_\ell b_\ell+(a_\ell+b_\ell)(T_\ell-\xi_j^{(\ell)})+\sigma_2(\xi^{(\ell)}|j)
        =a_\ell b_\ell-T_\ell^{1/2}p_j^{(\ell)}+O(T_\ell).
\]
Therefore
\[
        \sum_{j\ge3}d_j^{(\ell)}u_j^{(\ell)}
        =a_\ell b_\ell T_\ell^{1/2}\sum_{j\ge3}v_j^{(\ell)}
        -T_\ell p^{(\ell)}\cdot v^{(\ell)}+O(T_\ell^{3/2}).
\]
Using \eqref{eq:app-d1-d2-limits-ranktwo}, the constraint $d^Tz=0$ becomes
\[
        0=T_\ell(b-h+o(1))x_\ell+T_\ell(a-h+o(1))y_\ell
        +a_\ell b_\ell T_\ell^{1/2}\sum_{j\ge3}v_j^{(\ell)}
        -T_\ell p^{(\ell)}\cdot v^{(\ell)}+o(T_\ell).
\]
Thus $\sum_{j\ge3}v_j^{(\ell)}=O(T_\ell^{1/2})$.  Define
\begin{equation}\label{eq:app-mu-def}
        \mu_\ell:=\frac{1}{T_\ell^{1/2}}\sum_{j\ge3}v_j^{(\ell)}
        =\frac{1}{T_\ell}\sum_{j\ge3}u_j^{(\ell)}.
\end{equation}
After passing to a further subsequence, $\mu_\ell\to\mu$, and dividing the constraint by $T_\ell$ gives
\begin{equation}\label{eq:app-constraint-limit-ranktwo}
        (b-h)x+(a-h)y+ab\mu-p\cdot v=0.
\end{equation}
Also
\begin{equation}\label{eq:app-v-sum-zero}
        \sum_{j\ge3}v_j=0.
\end{equation}
A direct expansion gives, using \eqref{eq:app-mu-def},
\begin{align*}
        (z^{(\ell)})^TA^{(1)}z^{(\ell)}={}&
        2T_\ell y_\ell^2
        +2\sum_{j\ge3}(b_\ell+T_\ell-\xi_j^{(\ell)})(u_j^{(\ell)})^2
        -2T_\ell x_\ell y_\ell\\
        &-2x_\ell\sum_{j\ge3}(b_\ell+T_\ell-\xi_j^{(\ell)})u_j^{(\ell)}
        -2y_\ell\sum_{j\ge3}(a_\ell+T_\ell-\xi_j^{(\ell)})u_j^{(\ell)}\\
        &-2\sum_{3\le i<j\le n}(1-\xi_i^{(\ell)}-\xi_j^{(\ell)})u_i^{(\ell)}u_j^{(\ell)}.
\end{align*}
Dividing by $T_\ell$ and using \eqref{eq:app-x-y-sum} and \eqref{eq:app-v-sum-zero}, we obtain
\begin{equation}\label{eq:app-Q-limit-ranktwo}
        \frac{(z^{(\ell)})^TA^{(1)}z^{(\ell)}}{T_\ell}\to Q,
\end{equation}
where
\begin{equation}\label{eq:app-Q-def-ranktwo}
        Q:=2y^2-2xy+(2b+1)|v|^2-2(bx+ay)\mu+2p\cdot v.
\end{equation}
Recalling \eqref{eq:app-d1-d2-limits-ranktwo}, it remains to prove
\begin{equation}\label{eq:app-Q-target}
        Q\ge\frac32(b-h).
\end{equation}
Let
\[
        L:=bx+ay.
\]
Since $x+y=1$, the constraint \eqref{eq:app-constraint-limit-ranktwo} is
\[
        ab\mu=p\cdot v-L+h.
\]
Substituting $\mu$ into \eqref{eq:app-Q-def-ranktwo},
\[
        Q=2y^2-2xy+\frac{2L(L-h)}{ab}+(2b+1)|v|^2
        +2\left(1-\frac{L}{ab}\right)p\cdot v.
\]
Young's inequality and \eqref{eq:app-p-h-identities} give
\[
        (2b+1)|v|^2+2\left(1-\frac{L}{ab}\right)p\cdot v
        \ge
        -\frac{|p|^2}{2b+1}\left(1-\frac{L}{ab}\right)^2
        =-\frac{2h}{2b+1}\left(1-\frac{L}{ab}\right)^2.
\]
Thus
\begin{align*}
        Q-\frac32(b-h)
        \ge F(h):={}&2y^2-2xy+\frac{2L(L-h)}{ab}
        -\frac{2h}{2b+1}\left(1-\frac{L}{ab}\right)^2
        -\frac32(b-h).
\end{align*}
By \eqref{eq:app-h-bound-ranktwo}, it suffices to prove $F(h)\ge0$ for all $x+y=1$ and $0\le h\le ab$.  Since $F$ is affine in $h$, it is enough to prove $F(0)\ge0$ and $F(ab)\ge0$.  Direct computation gives
\begin{align*}
        F(0)
        ={}&\frac{2(1-2b+2b^2)}{b(1-b)}
        \left(y-\frac{b(3b-1)}{2(1-2b+2b^2)}\right)^2
        +\frac{b^2(7-6b)}{2(1-2b+2b^2)},
\end{align*}
\begin{align*}
        F(ab)
        ={}&\frac{4(2b^2-3b+2)}{(1-b)(2b+1)}
        \left(y-\frac{b(2b^2-b+1)}{2(2b^2-3b+2)}\right)^2
        +\frac{b^3(3-2b)}{2(2b^2-3b+2)}.
\end{align*}
Since $0<b<1$, both expressions are nonnegative.  Hence \eqref{eq:app-Q-target} holds, and therefore
\begin{equation}\label{eq:app-theta1-ranktwo}
        \liminf_{\ell\to\infty}\Theta_1(\lambda^{(\ell)})\ge\frac32.
\end{equation}
The proof for $m=2$ is the same after interchanging the first two coordinates, so
\begin{equation}\label{eq:app-theta2-ranktwo}
        \liminf_{\ell\to\infty}\Theta_2(\lambda^{(\ell)})\ge\frac32.
\end{equation}

\subsubsection{Small indices \texorpdfstring{$m\ge3$}{}}

Fix $m\ge3$.  Let $z^{(\ell)}$ be feasible and suppose
\[
        \frac{(z^{(\ell)})^TA^{(m)}(\lambda^{(\ell)})z^{(\ell)}}{d_m^{(\ell)}}\le C.
\]
By \eqref{eq:app-coercive}, and because
\[
        \sigma_1(\lambda^{(\ell)}|1m)\to b,
        \qquad
        \sigma_1(\lambda^{(\ell)}|2m)\to a,
        \qquad
        \sigma_1(\lambda^{(\ell)}|im)\to1
        \quad(i\ge3,\ i\ne m),
\]
all components except possibly $z_m^{(\ell)}$ are bounded.  The constraint $d^Tz=0$, together with \eqref{eq:app-dj-ranktwo}, then bounds $z_m^{(\ell)}$ as well.  Passing to a subsequence,
\[
        z^{(\ell)}\to z.
\]
Since $d_1^{(\ell)},d_2^{(\ell)}\to0$ and $d_j^{(\ell)}\to ab>0$ for $j\ge3$, the constraint $d^Tz=0$ gives
\[
        \sum_{j\ge3}z_j=0.
\]
Together with $e^Tz=1$, this gives
\begin{equation}\label{eq:app-z1-z2-sum-ranktwo}
        z_1+z_2=1.
\end{equation}
At the limiting vector $(a,b,0,\ldots,0)$, the quadratic form satisfies
\[
        \lim_{\ell\to\infty}(z^{(\ell)})^TA^{(m)}(\lambda^{(\ell)})z^{(\ell)}
        =2bz_1^2+2az_2^2+3\sum_{\substack{j\ge3\\ j\ne m}}z_j^2+z_m^2.
\]
The small-index terms are nonnegative.  Using \eqref{eq:app-z1-z2-sum-ranktwo},
\[
        2bz_1^2+2az_2^2\ge2ab.
\]
Since $d_m^{(\ell)}\to ab>0$ and $\sigma_1(\lambda^{(\ell)})=1$,
\begin{equation}\label{eq:app-thetam-ranktwo}
        \liminf_{\ell\to\infty}\Theta_m(\lambda^{(\ell)})\ge2,
        \qquad m\ge3.
\end{equation}
Combining \eqref{eq:app-theta1-ranktwo}, \eqref{eq:app-theta2-ranktwo}, and \eqref{eq:app-thetam-ranktwo},
\begin{equation}\label{eq:app-theta-ranktwo-final}
        \liminf_{\ell\to\infty}\Theta(\lambda^{(\ell)})\ge\frac32
\end{equation}
whenever the limiting vector is \eqref{eq:app-rank-two-limit}.

\subsection{Completion of the proof}

By \eqref{eq:app-theta-rankone-final} and \eqref{eq:app-theta-ranktwo-final}, Proposition \ref{prop:app-compactness} is proved.  We now prove Theorem \ref{thm:theta-k3}.  Fix $\delta$ with $1<\delta<3/2$.  If no such $\eps$ existed, then for every $\ell$ there would be $\lambda^{(\ell)}\in\Gamma_3$ such that
\[
        d_1^{(\ell)}\le \frac1\ell\sigma_2(\lambda^{(\ell)}),
        \qquad
        \Theta(\lambda^{(\ell)})<\delta.
\]
By homogeneity, normalize $\sigma_1(\lambda^{(\ell)})=1$.  Proposition \ref{prop:app-compactness} gives
\[
        \liminf_{\ell\to\infty}\Theta(\lambda^{(\ell)})\ge\frac32>\delta,
\]
a contradiction.  Hence an $\eps(n,\delta)>0$ exists.  This proves Theorem \ref{thm:theta-k3}.  In particular, one may take any fixed $1<\delta<3/2$, for instance $\delta=4/3$, after choosing $\eps>0$ sufficiently small.

\bibliographystyle{acm}
\bibliography{Reference}
\end{document}